\documentclass{amsart}

\usepackage{amsmath}
\usepackage[T1]{fontenc}
\usepackage[utf8]{inputenc}
\usepackage[english]{babel}
\usepackage{csquotes}
\usepackage{amsfonts}
\usepackage{amsthm}
\usepackage{amssymb}
\usepackage{enumitem}
\usepackage{xcolor}
\usepackage{mathabx}
\usepackage{graphicx}
\usepackage{changepage}
\usepackage{bbold}
\usepackage{tikz}
\usepackage{tikz-cd}
\usepackage{tabularx}
\usepackage{mathtools}
\usepackage{bm}
\usepackage{hyperref}
\usepackage[firstinits=true]{biblatex}  %Imports biblatex package
\usetikzlibrary{arrows}

\usepackage{pgfplots}
\pgfplotsset{width=10cm,compat=1.9}
\usepgfplotslibrary{external}

\DeclareRobustCommand{\gobblefive}[5]{}
\newcommand*{\SkipTocEntry}{\addtocontents{toc}{\gobblefive}}

\newtheorem{theorem}{Theorem}[section]
\newtheorem*{theorem**}{Theorem\theoremnum}
\newenvironment{theorem*}[1][]{%
	\edef\theoremnum{\if\relax\detokenize{#1}\relax\else~#1\fi}% Store theorem number
	\begin{theorem**}
	}{%
	\end{theorem**}
}  
\newtheorem{pro}{Proposition}[section]
\newtheorem{cor}{Corollary}[section]
\newtheorem{lem}{Lemma}[section]
\theoremstyle{definition}
\newtheorem{definition}{Definition}[section]
\newtheorem{ex}{Example}[section]

\theoremstyle{remark}
\newtheorem{rem}{Remark}

\hypersetup{
	colorlinks=true,
	linkcolor=blue,
	filecolor=magenta,      
	urlcolor=blue,
}

\graphicspath{ {./images/} }

\author{Riccardo Ontani} 
\email{rontani@sissa.it}
\address{SISSA, via Bonomea 265, 34136 Trieste, Italy} 

\title{Virtual invariants of critical loci in GIT quotients of linear spaces.}

\begin{document}
	\begin{abstract}
		We use an equivariant version of the localization formula of Jeffrey and Kirwan to prove a formula for virtual invariants (DT, $\chi_y$, $\text{Ell}$) of critical loci in quotients of linear spaces by actions of reductive algebraic groups. 
		
		In particular we recover formulae for the invariants of critical loci of potentials in moduli spaces of quiver representations predicted by physicists.
	\end{abstract}
	\maketitle	
	\tableofcontents
	\section{Introduction.}
	In the last few years, in the physics literature there has been an increasing appearance of applications of the localization formula of Jeffrey and Kirwan \cite{JeffreyKirwan} to compute, via supersymmetric localization, quantities of physical interest. The main aim of this work is to interpret and to put on rigorous mathematical ground the results of \cite{HoriBenini, Cordova}, in which the invariant of interest is the \textit{elliptic genus} of two dimensional gauge theories. In \cite{BMP}, the authors specialize the results of \cite{HoriBenini} to quiver theories and they compute the \textit{Witten index} of such theories. In both cases, the relevant quantities $Q$ are expressed by the following formula:
	\begin{align*}
		Q := k \sum_{P} \text{JK}_P^{\mathcal{W}_P} (Z, \eta),
	\end{align*}
	where $k$ is some constant and $\text{JK}_P^{\mathcal{W}_P}$, called \textit{Jeffrey-Kirwan residue}, is an operation defined in \cite{SzenesVergne}. Here $\eta$ is a stability parameter (called \textit{Fayet-Iliopulos parameter} in physics), $Z$ is a suitable meromorphic function on some abelian Lie algebra and the sum ranges over some of the poles of such function. In \cite{HoriBenini}, everything is computed via supersymmetric localization to reduce the path integral, originally defined as an integral over an infinite dimensional manifold, to an integral over a finite dimensional one. After some careful manipulation the finite dimensional integral is shown to be a sum of contributions that, a posteriori, can be written as Jeffrey-Kirwan residues of $Z$. In other words, in \cite{HoriBenini} the geometric localization formula of Jeffrey and Kirwan doesn't appear in a direct way, but the interesting invariant $Q$ can be written in a form that resembles the output of such localization procedure.\\
	In this work we take a purely mathematical point of view by applying the geometric Jeffrey-Kirwan localization formula. We consider
	\begin{enumerate}
		\item a reductive group $G$ acting on a linear space $V$,
		\item a linearization $\xi$ so that the quotient $\mathcal{A}:= V/\!/G$ is a smooth orbifold,
		\item an action of $\mathbb{C}^\ast$ on $V$ descending to one on $\mathcal{A}$ so that $\mathcal{A}^{\mathbb{C}^\ast}$ is proper and
		\item a potential function $\varphi : \mathcal{A} \rightarrow \mathbb{C}$ of degree $d \in \mathbb{Z}\setminus \lbrace 0 \rbrace$ w.r.t. the $\mathbb{C}^\ast$-action.
	\end{enumerate}  
	 We are interested in the ($\mathbb{C}^\ast$-equivariant) virtual invariants of the critical locus $X:= \text{Crit}(\varphi)$, endowed with the usual perfect obstruction theory. The invariants we care about are
	\begin{itemize}
		\item the \textit{Donaldson-Thomas} invariant $\text{DT}$,
		\item the \textit{virtual Hirzebruch genus} $\chi_y$ and
		\item the \textit{virtual chiral elliptic genus} $\text{Ell}(q,y)$,
	\end{itemize} 
	defined via virtual localization in the nonproper case (see definitions \ref{DT}, \ref{Chiy} and \ref{Ell}). Notice that there are no further assumptions on the potential, which can even be the trivial one $\varphi=0$. The deformation invariance of these invariants ensures that the result does not depend on the particular potential chosen (but it does depend on its degree $d$).
	It's worth remarking that there are different, non equivalent ways to define virtual invariants which coincide once one restricts to the case of proper varieties. In particular the virtual invariants introduced above, defined via virtual localization, agree with the corresponding invariants defined via Behrend localization if $\text{Crit}(\varphi)$ is proper \cite{ThomasJiang}.
	
	It turns out that, on the Lie algebra $\mathfrak{t}$ of a maximal subtorus $T\subseteq G$, there are some affine functions $f_i$ whose zero loci form an important hyperplane arrangement. The 0-dimensional intersections of all possible subsets of this family of hyperplanes form a finite set $\mathfrak{M}$, admitting a subset of $\xi$-stable elements $\mathfrak{M}^\xi$. If we denote with $\alpha \in \mathfrak{t}^\vee$ the roots of $G$, with $\rho \in \mathfrak{t}^\vee$ the weights for the $T$-action on $V$ and with $R \in \mathbb{Z}^{\text{dim}V}$ the weights for $\mathbb{C}^\ast \curvearrowright V$, we can build three meromorphic functions
	\begin{align*}
		Z_s, Z_\hbar, Z_{\tau, \hbar} : \mathfrak{t} \dashrightarrow \mathbb{C}
	\end{align*}
	as in (\ref{functions}). Let $\tilde{\xi}$ be a suitably generic perturbation of $\xi$, thought as an element of $\mathfrak{t}_\mathbb{R}^\vee$, inside the same chamber ($\tilde{\xi}$ is called a sum-regular perturbation, see definition \ref{regularStability}). The main result of this work is that we can express our virtual invariants (up to a factor $\vert W \vert$ where $W$ is the Weyl group of $G$) in terms of Jeffrey-Kirwan residues of these functions:
	\begin{theorem*}[\ref{mainTheorem}]
		The virtual invariants of $X$ can be computed with the following formulae:
		\begin{align*}
			\text{DT} &= \frac{1}{\vert W \vert} \sum_{P \in \mathfrak{M}^\xi} \text{JK}_P^{\mathcal{W}_P}\left(Z_s, \tilde{\xi}\right), \qquad \forall s \in \mathbb{C}^\ast\\
			\chi_{e^{2 \pi i \hbar}} &= \frac{1}{\vert W \vert} \sum_{P \in \mathfrak{M}^\xi} \text{JK}_P^{\mathcal{W}_P}\left(Z_\hbar, \tilde{\xi}\right),\\
			\text{Ell}(e^{2 \pi i \tau}, e^{2\pi i \hbar}) &= \frac{1}{\vert W \vert} \sum_{P \in \mathfrak{M}^\xi} \text{JK}_P^{\mathcal{W}_P}\left(Z_{\tau, \hbar}, \tilde{\xi}\right).
		\end{align*}
	\end{theorem*}
	\begin{rem}
		Notice that other works on Jeffrey-Kirwan localization use a different notation, where the sum of residues over $P \in \mathfrak{M}^\xi$ is packed together into a single JK residue. We will keep the sum over $P$ explicit throughout the paper.
	\end{rem}
	The precise statement of the result is given in the self contained section \ref{statement}.
	The structure of the paper is the following:
	\begin{itemize}
		\item in section \ref{howToUse} we briefly explain the formula and how to use it from a practical point of view. This section can be useful to the reader that directly wants to understand how to apply the result without having to go through all the proofs in the paper.
		\item In section \ref{sectionLinear} we introduce some notation and we study the structure of the fixed locus of $\mathbb{C}^\ast$-actions on toric quotients of linear spaces. In particular, in proposition \ref{fixedTlocus} we establish a bijection between a subset $\mathfrak{M}^\xi$ of $\mathfrak{M}$ and the connected components of $(V/\!/T)^{\mathbb{C}^\ast}$. 
		\item In section \ref{sectionJK} we recall the definition of JK residue and we describe a $\mathbb{C}^\ast$-equivariant version of the Jeffrey-Kirwan localization formula, preparing the ground for applications to our problem.
		\item In section \ref{sectionVir} we first recall the definitions of the invariants we want to compute. Then we apply the virtual localization formula of \cite{Graber} and we push forward the computations onto the ambient space $\mathcal{A}$. There, we apply the results of section \ref{sectionJK} to pull back the computations onto $(V/\!/T)^{\mathbb{C}^\ast}$, which we know well from section \ref{sectionLinear}. Finally we apply JK localization to express the invariants as JK residues, proving the main theorem \ref{mainTheorem}.
		\item In section \ref{sectionExamples} we apply this result to find formulae for classical invariants of projective varieties and complete intersections therein, reproducing the results in \cite{HoriBenini}. Furthermore, we specialize the formula to the case of quiver varieties, obtaining in particular the result of \cite{BMP}.
	\end{itemize}
	\SkipTocEntry\subsection*{Acknowledgments}
	I wish to thank Richard Thomas for patiently teaching me all I know about virtual invariants and localization techniques. I'm grateful to Jacopo Stoppa, Sohaib Khalid and Pierre Descombes for many interesting discussions related to the paper.
	\SkipTocEntry\section*{Notation.}
	\textbf{K-theory.}
	Let $X$ be a scheme with a $\mathbb{C}^\ast$-action. The $K$-group of equivariant coherent sheaves on $X$ is denoted with $K_0^{\mathbb{C}^\ast}(X)$ and it is a module over $K^0_{\mathbb{C}^\ast}(X)$, the $K$-ring of equivariant vector bundles on $X$, via tensor product. We will denote with $\mathfrak{s}$ the trivial bundle on $X$ having action of weight 1 on the fibers, $y$ will denote its $K$-theory class and $s$ denotes its first Chern class.
	Given a vector bundle $V$ on $X$ we will denote with $\Lambda^p V$ the $p^{\text{th}}$ exterior power of $V$ and with $S^pV$ the $p^{\text{th}}$-symmetric power. We set
	\begin{align*}
		\Lambda_q V &:= \sum_{p=0}^\infty q^p \cdot [\Lambda^p V] \in K^0_{\mathbb{C}^\ast}(X)[q],\\
		\text{Sym}_q V &:= \sum_{p=0}^\infty q^p \cdot [S^p V] \in K^0_{\mathbb{C}^\ast}(X)[\![q]\!]
	\end{align*}
	for every $V \in K^0_{\mathbb{C}^\ast}(X)$. Since for coherent sheaves $\Lambda (V \oplus W) \simeq (\Lambda V) \otimes (\Lambda W)$ and $\text{Sym} (V \oplus W) \simeq (\text{Sym} V) \otimes (\text{Sym} W)$, these constructions convert sums into products. Recall that for $V \in K^0_{\mathbb{C}^\ast}$ we have
	\begin{align*}
		\text{Sym}_q(V) \cdot \Lambda_{-q}(V) = 1.
	\end{align*}
	\textbf{Fixed loci.}
	The fixed locus of $X$ is denoted with $X^{\mathbb{C}^\ast}$. Many times we will write down formulae as if $X^{\mathbb{C}^\ast}$ is connected. If it is not, the summation symbol over the connected components is left unwritten in order to simplify the notation. If $X$ is not proper but the fixed locus $X^{\mathbb{C}^\ast}$ is, we define the equivariant pushforwards in cohomology and K-theory by localization
	\begin{align*}
		\int_X \alpha := \int_{X^{\mathbb{C}^\ast}} \frac{\alpha}{e^{\mathbb{C}^\ast}(\mathcal{N}_{X^{\mathbb{C}^\ast}/X})}, \qquad \chi(X, V):= \chi\left(X^{\mathbb{C}^\ast}, \frac{V}{\Lambda_{-1}\mathcal{N}_{X^{\mathbb{C}^\ast}/X}^\vee}\right).
	\end{align*}
	Notice that, if $X$ is proper, the localization formula of \cite{Graber} ensures that these are the usual equivariant integral and equivariant Euler characteristic.\\
	\textbf{Lie algebras}.
	Given a reductive algebraic group $G$, the corresponding complex Lie algebra is denoted with $\mathfrak{g}$. If $T$ is a torus, we denote with $\mathfrak{t}_\mathbb{Z}$ the integral lattice in $\mathfrak{t}$. Its real Lie algebra (the Lie algebra of the unique maximal compact subgroup) is denoted with $\mathfrak{t}_\mathbb{R}$.\\
	\textbf{Eta and theta functions.}
	The \textit{Dedekind eta function} is the formal power series $\eta \in q^{\frac{1}{24}} \cdot \mathbb{Z}[\![q]\!]$ given by
	\begin{align*}
		\eta(q) = q^{\frac{1}{24}} \prod_{n\geq 1} (1-q^n).
	\end{align*}
	The Jacobi theta function is the power series $\theta \in q^{\frac{1}{8}} y^{-\frac{1}{2}} \cdot \mathbb{C}[\![y,q]\!]_y$
	\begin{align*}
		\theta(q;y) := -iq^{\frac{1}{8}}(y^{\frac{1}{2}}- y^{-\frac{1}{2}}) \prod_{n\geq 1} (1-q^n) (1-y q^n) (1-y^{-1}q^n).
	\end{align*}
	We also denote with the same Greek letters the functions
	\begin{align*}
		\eta(\tau) := \eta(e^{2\pi i \tau}), \qquad
		\theta(\tau \vert z) := \theta(e^{2\pi i \tau}, e^{2\pi i z}),
	\end{align*}
	which enjoy nice modular properties \cite{MumfordTheta}.
	\section{How to use the formula in practice.}\label{howToUse}
	Assume one wants to compute the virtual invariants of the critical locus of a potential inside some quotient $V/\!/G$ of a linear space by a reductive group. Assume to have the the following data:
	\begin{enumerate}
		\item a reductive group $G$, with Weyl group $W$, acting on a linear space $\mathbb{C}^N$ so that the action of the maximal subtorus $T \simeq (\mathbb{C}^\ast)^k$ is diagonal:
		\begin{align*}
			(t \cdot v)_i = t^{\rho_i} v_i \qquad \forall i \in \lbrace 1, ..., N \rbrace
		\end{align*}
		for some $\rho_1, ..., \rho_N \in \mathbb{Z}^k$.
		\item The roots $\alpha_1, ..., \alpha_M \in \mathfrak{t}_\mathbb{Z}^\vee \simeq \mathbb{Z}^k$ of the group $G$ (namely the nontrivial weights of the adjoint representation $\mathfrak{g}$)
		\item a regular linearization $\xi \in \mathbb{Z}^k$ so that the quotient $V/\!/G$ is a smooth orbifold. These words can be interpreted in many ways, depending on one's background.
		\begin{itemize}
			\item An algebraic geometer should think of it as the Weyl invariant element $\xi \in (\mathfrak{t}_\mathbb{Z}^\vee)^W \subseteq \mathbb{Z}^k$ defining the character of $G$ used to build the linearization (see remark \ref{linearizationsAndCharacters}). The regularity condition is equivalent to the fact that for the corresponding linearization there are not strictly semistable objects.
			\item A symplectic geometer should think of it as the $G$-invariant functional $\xi \in (\mathfrak{g}^\vee)^{G}$ on the Lie algebra of $G$ which we use to take the symplectic reduction. By restricting to the Cartan subalgebra $\mathfrak{t}$ one can think of $\xi$ as of an element of $(\mathfrak{t}_\mathbb{Z}^\vee)^W\subseteq \mathbb{Z}^k$. The regularity condition coincides with $\xi$ being a regular value for the momentum map corresponding to the $T$-action.
			\item A physicist should think of it as the Fayet-Iliopulos term $\xi \in (\mathfrak{t}^\vee_\mathbb{Z})^W \subseteq \mathbb{Z}^k$.
		\end{itemize}
		In all cases, the regularity condition corresponds to $\xi$ being in the interior of the fan spanned by the weights:
		\begin{align*}
			\xi \in \text{Span}_{\mathbb{R}_{\geq 0}}(\rho_1, ..., \rho_N) \setminus \bigcup_{i_1, ..., i_{k-1} = 1}^N \text{Span}_{\mathbb{R}_{\geq 0}}(\rho_{i_1}, ..., \rho_{i_{k-1}}).
		\end{align*}
		\item an action of $\mathbb{C}^\ast$ on $\mathbb{C}^N$
		\begin{align*}
			(s \cdot v)_i = s^{R_i} v_i \qquad \forall i \in \lbrace 1, ..., N \rbrace,
		\end{align*}
		where $R_1, ..., R_N \in \mathbb{Z}$,
		descending to one on $\mathbb{C}^N/\!/G$ so that $(\mathbb{C}^N/\!/G)^{\mathbb{C}^\ast}$ is proper. There is a nice criterion to check this condition directly from this data, provided that one knows the generators of the $G$-invariant functions on $\mathbb{C}^N$ (see remark \ref{projectivityOfInt}). 
		\item a $G$-invariant function $\tilde{\varphi} : \mathbb{C}^N \rightarrow \mathbb{C}$ of degree $d \in \mathbb{Z}\setminus \lbrace 0 \rbrace$ with respect to the $\mathbb{C}^\ast$-action, therefore inducing a degree $d$ potential function $\varphi$ on $\mathbb{C}^N/\!/G$.
	\end{enumerate}  
	Given $x \in \mathbb{Z}^k$ and $y \in \mathbb{C}^k$, denote with $x y$ the pairing $\sum_{l=1}^k x_k y_k$.
	Then one can build the meromorphic function on $\mathfrak{t} \simeq \mathbb{C}^k$:
	\begin{align*}
		Z(u) = \frac{1}{d^k} \prod_{j=1}^M \frac{\alpha_j u}{\alpha_j u + d} \prod_{i=1}^N \frac{d-R_i - \rho_i u}{R_i + \rho_i u}.
	\end{align*}
	Consider the hyperplanes of poles given by the weights:
	\begin{align*}
		H_i := \lbrace u \in \mathbb{C}^k \quad : \quad \rho_i u + R_i = 0 \rbrace \qquad \forall i \in \lbrace 1, ..., N \rbrace
	\end{align*}
	and consider the set $\mathfrak{M}\subset \mathfrak{t} \simeq \mathbb{C}^k$ of isolated intersections, namely the points $P \in \mathbb{C}^k$ that can be realized as the intersection of $k$ independent such hyperplanes $H_i$. 
	Given such a point $P \in \mathfrak{M}$, consider the set of weights whose corresponding hyperplanes contain $P$:
	\begin{align*}
		\mathcal{W}_P := \left\lbrace \rho_i \quad : \quad \rho_i u(P) + R_i = 0 \right\rbrace.
	\end{align*}
	The subset of $\xi$-stable intersections $\mathfrak{M}^\xi$ is the set of points $P \in \mathfrak{M}$ so that
	\begin{align*}
		\xi \in \text{Span}_{\mathbb{R}_{\geq 0}}(\mathcal{W}_P).
	\end{align*} 
	Assume that no hyperplane of poles coming from the roots of $G$, namely
	\begin{align*}
		K_j := \left\lbrace u \in \mathbb{C}^k \quad : \quad \alpha_j u + d = 0 \right\rbrace \qquad \forall j \in \lbrace 1, ..., M\rbrace,
	\end{align*}
	contains a stable isolated intersection:
	\begin{align*}
		P \notin K_j \qquad \forall P \in \mathfrak{M}^\xi \text{ and } \forall j \in \lbrace 1, ..., M\rbrace.
	\end{align*}
	Then the formula to compute the DT invariant of $X:= \text{Crit}(\varphi)$, defined through virtual localization, is
	\begin{align*}
		\text{DT} = \frac{1}{\vert W \vert} \sum_{P \in \mathfrak{M}^\xi} \text{JK}_P^{\mathcal{W}_P}(Z, \tilde{\xi})
	\end{align*}
	where $\tilde{\xi}$ is a small perturbation (called sum-regular perturbation, see definition \ref{regularStability}) of the parameter $\xi$ in $\mathbb{R}^k$. The procedure to take the JK residue is explained in the self contained section \ref{subsJKresidues}.
	The same exact formula holds for the other invariants: if we set
	\begin{align*}
		Z_\hbar(u) = \left(\frac{\pi \hbar}{\sin(d\pi\hbar)}\right)^k \prod_{j=1}^M \frac{\sin(\pi \hbar\alpha_j u)}{\sin(\pi\hbar(\alpha_j u + d))} \prod_{i=1}^N \frac{\sin(\pi\hbar(d-R_i - \rho_i u))}{\sin(\pi\hbar(R_i + \rho_i u))}
	\end{align*}
	we have 
	\begin{align*}
		\chi_{e^{2 \pi i \hbar}} = \frac{1}{\vert W \vert} \sum_{P \in \mathfrak{M}^\xi} \text{JK}_P^{\mathcal{W}_P}(Z_\hbar, \tilde{\xi}).
	\end{align*}
	On the other hand, for
	\begin{align*}
		Z_{\tau,\hbar}(u) = \left(\frac{2\pi\hbar\eta(\tau)^3}{\theta(\tau \vert d\hbar)}\right)^k \prod_{j=1}^M \frac{\theta(\tau \vert \hbar\alpha_j u)}{\theta(\tau \vert \hbar(\alpha_j u + d))} \prod_{i=1}^N \frac{\theta(\tau \vert \hbar(d-R_i - \rho_i u))}{\theta(\tau \vert \hbar(R_i + \rho_i u))},
	\end{align*}
	we obtain 
	\begin{align*}
		\text{Ell}(e^{2 \pi i \tau}, e^{2 \pi i \hbar}) = \frac{1}{\vert W \vert} \sum_{P \in \mathfrak{M}^\xi} \text{JK}_P^{\mathcal{W}_P}(Z_{\tau,\hbar}, \tilde{\xi}).
	\end{align*}
	\section{Actions on linear spaces and their quotients.}\label{sectionLinear}
	Here we introduce some notation and prove results about fixed loci of $\mathbb{C}^\ast$-actions on toric varieties.
	\subsection{Circle actions on quotients of linear spaces.}\label{subsectionQuotientsofLinearSpaces}
	Here we recall some basic facts about GIT for linear spaces.
	Consider a complex linear space $V$ with a linear effective action of a reductive algebraic group $G$ having $W$ as Weyl group. Let $T$ be the maximal torus in $G$ and denote the corresponding complex Lie algebras with $\mathfrak{t} \subseteq \mathfrak{g}$. A linearization $\mathcal{L}$ for the action is given by the choice of a character $\chi : G \rightarrow \mathbb{C}^\ast$.
	Characters can be reconstructed from their restriction to the maximal torus and a character of the torus can be extended to $G$ if and only if it is $W$-invariant, hence
	\begin{align*}
		\text{Hom}(G, \mathbb{C}^\ast) \simeq \text{Hom}(T, \mathbb{C}^\ast)^W \simeq (\mathfrak{t}^\vee_\mathbb{Z})^W.
	\end{align*}
	\begin{rem}\label{linearizationsAndCharacters}
		Given a character $\chi : G \rightarrow \mathbb{C}^\ast$, we adopt the convention that the associated linearization $\mathcal{L}\rightarrow V$ is the equivariant bundle $V \times \mathbb{C}$ with action given by 
		\begin{align*}
			g \cdot (v,z) := \left(g\cdot v, \chi^{-1}(g) z\right).
		\end{align*}
		The minus sign at the exponent of $\chi$ in this notation is chosen so that this sign disappears in some conditions on $\xi$ that will appear later.
	\end{rem}
	\begin{definition}\label{regularLinearization}
		We say $\xi$ is \textit{regular} if and only if $V^{G-\text{sst}} = V^{G-\text{st}}$ for the corresponding linearization, so that the GIT quotient $V/\!/G$ is a smooth orbifold.
	\end{definition}
	\begin{rem}
		Notice that $\xi$ is regular if and only if the induced linearization for the $T$-action is regular:
		\begin{align*}
			V^{G-sst} = V^{G-st} \iff V^{T-sst} = V^{T-st}
		\end{align*}
		as shown in \cite{Dolgachev} via the Hilbert-Mumford criterion.
	\end{rem}
	Consider the decomposition of $V$ into $T$-invariant 1-dimensional spaces $V_i \simeq \mathbb{C}$
	\begin{align}\label{weightSpaceDec}
		V \simeq \bigoplus_{i=1}^{\text{dim}V} V_i,
	\end{align}
	so that $T$ acts through a morphism to $(\mathbb{C}^\ast)^{\text{dim}V}$.
	Denoting with $\mathcal{W} \subset \mathfrak{t}_\mathbb{Z}^\vee$ the set of weights of $T \curvearrowright V$, we have a function
	\begin{align*}
		\rho: \lbrace 1, ..., \text{dim}V \rbrace \rightarrow \mathcal{W} \quad : \quad i \mapsto \rho_i,
	\end{align*}
	which associates to the index $i$ the weight $\rho_i$ of $T \curvearrowright V_i$.
	\begin{definition}
		A \textit{R-charge} for $V$ is a vector $R \in \mathbb{Z}^{\text{dim}V}$. This induces an action of $\mathbb{C}^\ast$ on $V$ via the decomposition (\ref{weightSpaceDec}) by:
		\begin{align}\label{CircleAction}
			(s \cdot v)_i := s^{R_i} v_i.
		\end{align}
	\end{definition}
	It will be important to us the case in which the $\mathbb{C}^\ast$-action defined by $R$ commutes with the one of $G$. 
	In this case, the action of ${\mathbb{C}^\ast}$ descends to an action on the quotients $V/\!/T$ and $V/\!/G$.
	\begin{rem}\label{CommutationExample}
		This happens, for example, in the following case. Consider the decomposition of $V$ into a direct sum of subrepresentations of $G$:
		\begin{align}\label{irreducibleDec}
			V \simeq \bigoplus_{\lambda \in \Lambda} V_\lambda
		\end{align}
		We can define function $\lambda : \lbrace 1, ..., \text{dim}V \rbrace \rightarrow \Lambda$ associating to every index $i$ the element $\lambda(i)$ such that $V_i \subseteq V_{\lambda(\rho)}$. In this case we can define the ${\mathbb{C}^\ast}$-action by choosing $\tilde{R} \in \mathbb{Z}^\Lambda$ and setting
		\begin{align*}
			R_i := \tilde{R}_{\lambda(i)}.
		\end{align*}
		This means that the ${\mathbb{C}^\ast}$-action respects the decomposition above, acting on $V_\lambda$ via scalar multiplication by $s^{\tilde{R}_\lambda}$. In this case the ${\mathbb{C}^\ast}$-action commutes with the one of $G$ since $G$ acts linearly on each $V_\lambda$.
	\end{rem} 
	\subsection{Combinatorics of the fixed loci loci in toric varieties.}\label{subsCombinatorics}
	Assume a torus $T$ acts on a linear space $V$ with a regular linearization $\xi \in \mathfrak{t}^\vee_\mathbb{Z}$. Consider a decomposition of $V$ into irreducible representations as in (\ref{weightSpaceDec}) and a $\mathbb{C}^\ast$-action (defined by a given $R$-charge $R \in \mathbb{Z}^{\text{dim}V}$) on $V$ commuting with $T$, which induces a $\mathbb{C}^\ast$-action on the quotient $V/\!/T$. The Atiyah-Bott localization formula is very useful to compute integrals on $V/\!/T$ by exploiting the $\mathbb{C}^\ast$-action, but in order to use this tool one needs to know the connected components of the fixed locus on $V/\!/T$. We will encounter this picture later in the paper, so for the moment we take some time to carefully describe this fixed locus. It turns out that the connected components of the $\mathbb{C}^\ast$-fixed locus on the toric variety $V/\!/T$ can easily be described by combinatorics. The combinatorial objects that we have to consider are the functions
	\begin{align*}
		f_i : \mathfrak{t} \rightarrow \mathbb{C} \quad : \quad f_i(u) := \rho_i(u)+R_i,
	\end{align*}
	indexed by $i \in \lbrace 1, ..., \text{dim}V\rbrace$, where $\rho_i$ is the weight of the $T$-action on the $i^{th}$ indecomposable piece $V_i \subseteq V$. With these functions we can give the following
	\begin{definition}\label{isolatedInt}
		The set $\mathfrak{M} \subset \mathfrak{t}$ of \textit{isolated intersections} of the hyperplanes $V(f_i)$ is the set of points $P \in \mathfrak{t}$ at which at least $\text{dim}\mathfrak{t}$ independent hyperplanes vanish:
		\begin{align*}
			\mathfrak{M} := \left\lbrace P \in \mathfrak{t} \,\, : \,\, \exists I_P \subseteq \lbrace 1, ..., \text{dim}V\rbrace \text{ s.t. } \bigcap_{i \in I_P} V(f_i) = \lbrace P \rbrace \right\rbrace.
		\end{align*}
	\end{definition}
	Given $P \in \mathfrak{M}$, we set  $I_P := \lbrace i \in \lbrace 1, ..., \text{dim}V\rbrace \,\, : \,\, f_i(P)=0 \rbrace$ and let $\mathcal{W}_P$ be the set of the corresponding weights: $\mathcal{W}_P = \rho(I_P)$. In other words, a weight $\rho_i$ belongs to $\mathcal{W}_P$ if and only if $\rho_i(P) + R_i = 0$.
	\begin{definition}\label{stableIsolatedInt}
		The subset $\mathfrak{M}^\xi \subseteq \mathfrak{M}$ of \textit{stable isolated intersections} with respect to the regular stability $\xi$ is
		\begin{align*}
			\mathfrak{M}^\xi:= \left\lbrace P \in \mathfrak{M} \,\, : \,\, \xi \in \text{span}_{\mathbb{R}_{\geq 0}}(\mathcal{W}_P) \right\rbrace.
		\end{align*}
	\end{definition}
	The main point we want to prove is the following. 
	\[
		\begin{gathered}
			\text{The connected components of }(V/\!/T)^{\mathbb{C}^\ast} \text{ are in}\\
			\text{bijection with the stable isolated intersections }\mathfrak{M}^\xi.
		\end{gathered}
	\]
	Here we describe how to build a connected component of the ${\mathbb{C}^\ast}$-fixed locus from such isolated intersections.
	For each isolated intersection $P \in \mathfrak{M}^\xi$ consider the subvariety
	\begin{align*}
		\mathcal{B}_P := \bigcap_{j \in \lbrace 1, ..., \text{dim}V \rbrace \setminus I_P} V(v_j) \subseteq V/\!/T.
	\end{align*}
	where $v_j$ is the coordinate on the $j^{th}$ irreducible bit $V_j \subseteq V$.
	In other words, $\mathcal{B}_P$ is the toric variety built as the quotient by $T$ of the subspace of $V$ given by
	\begin{align*}
		V_P := \bigoplus_{i \in I_P} V_i.
	\end{align*}
	\begin{lem}\label{connectedFixedLocus}
		The subvariety $\mathcal{B}_P \subseteq V/\!/T$ is not empty, it is connected and fixed by ${\mathbb{C}^\ast}$.
	\end{lem}
	\begin{proof}
		By definition $\mathcal{B}_P$ is the toric variety built as GIT quotient of $V_P$ with respect to the linearization $\xi \in \mathfrak{t}^\vee$. Since this $\xi$ is inside the momentum cone given by the positive span of weights, the quotient is nonempty and it is connected. Let $x \in \mathcal{B}_P$ and consider a representative $v \in V$. The point $x$ is fixed by ${\mathbb{C}^\ast}$ if and only if, for every $s \in {\mathbb{C}^\ast}$, there is an element $t \in T$ such that such that
		\begin{align}\label{fixedCondition}
			t_i = s^{R_i} \qquad \forall i \in I_P.
		\end{align}
		Since the affine hyperplanes $\lbrace V(f_i) \, : \, i \in I_P \rbrace$ meet at the isolated point $P$, we can extract a subset $B$ of $I_P$ whose corresponding weights form a basis of $\mathfrak{t}^\vee_\mathbb{Q}$.		For every $j \in \lbrace 1, ..., \text{dim}V\rbrace$, we can write
		\begin{align*}
			\rho_j = \frac{1}{d_j} \sum_{i \in B}c_{i,j} \rho_i
		\end{align*}
		for $d_j, c_{i,j} \in \mathbb{Z}$. Up to setting $d:= \prod_{j} d_j$ and rescaling the coefficients $c_{i,j}$, we can assume that $d:= d_j$ doesn't depend on $j$. Let $\hat{s} \in S$ be a $d^{th}$-root of $s$, so that $s = \hat{s}^d$. Then we can set
		\begin{align*}
			t_j := \hat{s}^{\sum_{i \in B}c_{i, j} R_i}.
		\end{align*}
		Since $(\sum_{i \in B}c_{i, j} R_i)_j \in \mathbb{C}^{\text{dim}V}$ belongs to $\mathfrak{t}$ by definition, $t \in (\mathbb{C}^\ast)^{\text{dim}V}$ is an element of $T$. Finally, we have to check that this $t$ satisfies condition (\ref{fixedCondition}). For every $j \in I_P$ we have to prove that
		\begin{align*}
			\sum_{i \in B} c_{j,i}R_i = d R_j.
		\end{align*}
		Since $\rho_j= \sum_{i \in \mathcal{B}} c_{i,j} \rho_i$ we have that
		\begin{align*}
			d f_j - \sum_{i \in {B}} c_{i,j} f_i = dR_j - \sum_{i \in {B}} c_{i,j} R_i,
		\end{align*}
		but evaluating this at $P$ we find that the left hand side is zero.
	\end{proof}
	In order to show that these subvarieties are the connected components of $(V/\!/T)^{\mathbb{C}^\ast}$ we still have to prove that they are disjoint and maximal:
	\begin{lem}
		If $P, P^\prime \in \mathfrak{M}^\xi$ are different stable isolated intersections, the two fixed varieties $\mathcal{B}_P$ and $\mathcal{B}_{P^\prime}$ are disjoint. Moreover, if $x \in V/\!/T$ is fixed by ${\mathbb{C}^\ast}$, then there is $P \in \mathfrak{M}^\xi$ such that $x \in \mathcal{B}_P$.
	\end{lem}
	\begin{proof}
		Let $v \in \bigoplus_{i \in I_P \cap I_{P^\prime}} V_i$. We claim that it is not stable for the $T$-linearization given by $\xi$ by showing that its stabilizer is not finite. By hypothesis, the hyperplanes
		\begin{align*}
			V(f_i)\subset \mathfrak{t} \qquad i \in I_P \cap I_{P^\prime}
		\end{align*}
		all meet in the line passing by $P$ and $P^\prime$, hence $\mathcal{W}_P \cap \mathcal{W}_{P^\prime}$ don't span the whole $\mathfrak{t}^\vee$ which, thanks to lemma \ref{trivialCircle} in the appendix, implies that there is a $\mathbb{C}^\ast$ inside $T$ acting trivially on $v$. Let now $x$ be a fixed point for the ${\mathbb{C}^\ast}$-action on $V/\!/T$ and let $v \in V$ be a representative. Set
		\begin{align*}
			J^x := \left\lbrace j \in \lbrace 1, ..., \text{dim}V\rbrace  \,\, : \,\, v_i \neq 0 \right\rbrace
		\end{align*} 
		and consider the corresponding set of weights $\mathcal{W}^x \subset \mathcal{W}$ and functions $f_j: \mathfrak{t} \rightarrow \mathbb{C}$ for $j \in J^x$. These functions are all compatible with each other in a way that we now describe. Let $I \subseteq J^x$ induce a maximal $\mathbb{Q}$-linearly independent subset of $\mathcal{W}^x$. Then we claim that
		\begin{align}\label{compatibilityCondition}
			\bigcap_{j \in J^x} V(f_j) = \bigcap_{i \in I} V(f_i).
		\end{align}
		Assuming this holds, we complete $I$ to a subset $B \subset \lbrace 1, ..., \text{dim}V\rbrace$ whose weights form a basis of $\mathfrak{t}^\vee_\mathbb{Q}$ and find
		\begin{align*}
			\bigcap_{j \in J^x \cup B} V(f_j) = \bigcap_{k \in B} V(f_k) = \lbrace P \rbrace.
		\end{align*}
		This means that $x \in \mathcal{B}_P$ by definition of this variety.\\
		Let's prove claim (\ref{compatibilityCondition}). For every $j \in J^x$ we can write $\rho_j = \frac{1}{d}\sum_{i \in I} c_{i,j} \rho_i$ for $d, c_{i,j} \in \mathbb{Z}$ hence
		\begin{align*}
			df_j - \sum_{i \in I} c_{i,j} f_i = dR_j - \sum_{i \in I} c_{i,j}R_i \in \mathbb{Z}
		\end{align*}
		and we claim that this constant is zero. Since $x$ is fixed we have that for every $s \in {\mathbb{C}^\ast}$ there is a $t \in T$ such that for every $j \in J^x$
		\begin{align*}
			t_j v_j = s^{R_j} v_j, \qquad \text{hence} \qquad t_j = s^{R_j}.
		\end{align*}
		Since in $T$ we have that $t_j^d = \prod_{i \in I}t_i^{c_{i,j}}$ we have that
		\begin{align*}
			d R_j  = \sum_{i \in I} c_{i,j}R_i,
		\end{align*}
		concluding the proof of the claim, so equality (\ref{compatibilityCondition}) holds true.
	\end{proof}
	We have finally described combinatorially the connected components of the subscheme of fixed points for the ${\mathbb{C}^\ast}$-action on $V/\!/T$:
	\begin{pro}\label{fixedTlocus}
		The connected components of $(V/\!/T)^{\mathbb{C}^\ast}$ are in bijection with the stable isolated intersections $\mathfrak{M}^\xi$ via the function
		\begin{align*}
			\mathfrak{M}^\xi \rightarrow \lbrace \text{connected components of }(V/\!/T)^{\mathbb{C}^\ast} \rbrace \quad : \quad P \mapsto \mathcal{B}_P.
		\end{align*}
		If for every such connected component $F$ we consider the set
		\begin{align*}
			I_F:= \left\lbrace i \in \lbrace 1, ..., \text{dim}V\rbrace \,\, : \,\, F \nsubseteq V(v_i) \right\rbrace,
		\end{align*}
		the inverse map is given by
		\begin{align*}
			F \mapsto \bigcap_{i \in I_F} V(f_i).
		\end{align*}
	\end{pro}
	\begin{proof}
		In the previous lemmas we have shown that given $P$, the subvariety $\mathcal{B}_P$ is nonempty, connected and fixed by ${\mathbb{C}^\ast}$. Such varieties are disjoint one from the other and
		\begin{align*}
			(V/\!/T)^{\mathbb{C}^\ast} = \bigcup_{P \in \mathfrak{M}^\xi} \mathcal{B}_P,
		\end{align*}
		which makes it clear that the map $P \mapsto \mathcal{B}_P$ is a well defined bijection. In order to check that the latter map is the inverse of the former, notice that since $\mathcal{B}_P$ is the quotient of $V_P=\bigoplus_{i \in I_P} V_i$ we have that
		\begin{align*}
			I_{\mathcal{B}_P} := \left\lbrace i \in \lbrace 1, ..., \text{dim}V\rbrace \,\, : \,\, \mathcal{B}_P \nsubseteq V(v_i) \right\rbrace = I_P
		\end{align*}
		and $\bigcap_{i \in I_P} V(f_i) = \lbrace P \rbrace$.
	\end{proof}
	There is still one very important property that we have to study, which is the projectivity of the varieties $\mathcal{B}_P$.
	\begin{lem}\label{ProjectiveTLoci}
		Let $P \in \mathfrak{M}^\xi$ be a stable isolated intersection. Then $\mathcal{B}_P$ is projective if and only if the non-negative span of $\mathcal{W}_P$ is a strictly convex cone in $\mathfrak{t}^\vee$, or in the notation of section \ref{subsJKresidues}, if and only if $\mathcal{W}_P$ is a projective subset of $\mathfrak{t}^\vee$.
	\end{lem}
	\begin{proof}
		By definition $\mathcal{B}_P$ is the toric variety built as quotient of $V_P$ by $T$ for the linearization $\xi$. The set $\mathcal{W}_P$ is the set of weights of this representation and this is the usual criterion for the projectivity of toric varieties.
	\end{proof}
	This describes the reason behind the following
	\begin{definition}
		A stable isolated intersection $P \in \mathfrak{M}^\xi$ is called \textit{projective} if $\mathcal{W}_P$ spans a strictly convex cone.
	\end{definition}
	Let's give an example to fix the ideas:
	\begin{ex}
		Consider $\mathbb{P}^1$ built as quotient of $V \simeq \mathbb{C}^2$ by the scaling action of $T = \mathbb{C}^\ast$ with respect to the stability $\xi = 1$. Consider another action of $\mathbb{C}^\ast$ on $V$ given by the $R$-charge $(1,0)$, so $s \cdot (x_0, x_1) := (sx_0, x_1)$. This induces an action on $\mathbb{P}^1 \simeq V/\!/T$ and it's immediate to show that the only two fixed points are $[1:0]$ and $[0:1]$. Our prescription to find the fixed loci tells us the following. We first consider the Lie algebra $\mathfrak{t}\simeq \mathbb{C}$ of $T$ together with two functions
		\begin{align*}
			f_0(u) := u+1, \qquad f_1(u) := u.
		\end{align*}
		They give the hyperplane arrangement $\lbrace V(f_1), V(f_2) \rbrace$ and since the Lie algebra is one dimensional the set of hyperplanes coincides with its set of isolated intersections $\mathfrak{M} = \lbrace -1, 0 \rbrace$. Both intersections are $\xi$-stable, since $\xi=1$ is contained in the positive cone spanned by the weights (both $\rho_0, \rho_1\in \mathfrak{t}^\vee \simeq \mathbb{C}$ are equal to 1). The isolated intersection $-1$ is the vanishing locus of $f_0$, hence it corresponds to the subvariety $\mathcal{B}_{-1} = V(x_1)$ of $\mathbb{P}^1$, hence it corresponds to the fixed point $[1:0]$. Analogously $\mathcal{B}_0 = [0:1]$.
	\end{ex}
	\section{Preliminaries on Jeffrey-Kirwan localization.}\label{sectionJK}
	Localization formulae are useful tools to recover global information (like integrals of cohomology classes) in terms of local data at the fixed loci of some group action. The Jeffrey-Kirwan localization formula allows to integrate classes on quotients of spaces by looking at local data around the fixed loci of the action used to take the quotient. Here we describe some versions of this formula and we discuss a way to use it in the equivariant setting.
	\subsection{Jeffrey-Kirwan residues.}\label{subsJKresidues}
	Here and in the following section we recall some results of \cite{BrionVergne} and \cite{SzenesVergne} that make the computations appearing in the localization formula of Jeffrey and Kirwan \cite{JeffreyKirwan} easier to deal with.
	
	Let $\mathfrak{a}$ be a $n$-dimensional  real linear space and assume that the dual space $\mathfrak{a}^\vee$ is endowed with a lattice $\Gamma$ of full rank. Let $\mathfrak{A}$ be a \textit{projective} finite subset of $\Gamma$, namely a set whose positive linear span doesn't contain a line (or in other words the positive span is a strictly convex cone). 
	\begin{definition}\label{regularStability}
		An element $\xi \in \mathfrak{a}^\vee$ is called a \textit{regular stability} if there is no subset $S \subset \mathfrak{A}$ of cardinality $n-1$ so that $\xi \in \text{Span}_{\mathbb{R}_{\geq 0}}(S)$.
		If we set 
		\begin{align*}
			\Sigma \mathfrak{A} := \left\lbrace \sum_{w \in S} w \quad : \quad S \subset \mathfrak{A} \right\rbrace,
		\end{align*}
		we say that $\xi$ is \textit{sum-regular} if there is no subset $S \subset \Sigma\mathfrak{A}$ of cardinality $n-1$ so that $\xi \in \text{Span}_{\mathbb{R}_{\geq 0}}(S)$. If $\xi$ is a regular stability and $\tilde{\xi}$ is a sum regular stability so that the segment between them in $\mathfrak{a}^\vee$ is entirely made of regular stabilities, we say that $\tilde{\xi}$ is a \textit{sum-regular perturbation of $\xi$}.
	\end{definition}
	\begin{rem}
		It's easy to check that every regular stability admits a sum-regular perturbation.
	\end{rem}
	\begin{ex}\label{AbstractJKExample}
		Assume that a torus $T$ acts effectively on a complex linear space $V$, let $\mathcal{W}\subset \mathfrak{t}_\mathbb{Z}^\vee$ be the set of weights and choose a linearization $\xi \in \mathfrak{t}_\mathbb{Z}^\vee$. We can consider the real linear space $\mathfrak{t}_\mathbb{R}$, whose dual $\mathfrak{t}_\mathbb{R}^\vee$ is endowed with the full rank lattice $\Gamma:= \text{Span}_{\mathbb{Z}}\mathcal{W}$. It's well known \cite{Dolgachev} that
		\begin{enumerate}
			\item $\mathcal{W}$ is a projective subset of $\mathfrak{t}_\mathbb{R}^\vee$ if and only if there is a linearization such that $V/\!/T$ is projective.
			\item The stability $\xi$ is regular in the sense of Definition \ref{regularStability} if and only if the corresponding linearization is regular with respect to the projective set $\mathcal{W}$ in the sense of Definition \ref{regularLinearization}.
		\end{enumerate} 
	\end{ex}
	Pick any basis of the lattice $\Gamma$ and define $d\mu$ as the top form on $\mathfrak{a}$ given by the wedge product of the elements of the basis (this is well defined up to sign).
	Given a flag $F$ spanned by some elements of $\mathfrak{A}$, we can define the elements $\kappa_1, ..., \kappa_n \in \Gamma$ defined as
	\begin{align*}
		\kappa_i := \sum_{w \in \mathfrak{A}\cap F_i} w.
	\end{align*}
	\begin{definition}
		If these elements form a basis $\kappa := \lbrace \kappa_1, ..., \kappa_n \rbrace$ of $\mathfrak{a}^\vee$ we say that the flag $F$ is \textit{proper}. Fixed a sum regular stability $\xi$, if $\xi \in \text{span}_{\mathbb{R}_{>0}}(\kappa)$ we say that the flag is \textit{stable}. We denote with $\mathcal{F}(\mathfrak{A}, \xi)$ the set of all proper stable flags spanned by elements of $\mathfrak{A}$.
	\end{definition}
	There is a residue operation induced by every such flag:
	\begin{definition}
		Let $F \in \mathcal{F}(\mathfrak{A}, \xi)$ be a proper stable flag. The \textit{flag residue} of a meromorphic function $f$ on $\mathfrak{a}\otimes_\mathbb{R} \mathbb{C}$ is, up to a constant, the iterated residue computed with respect to $\kappa$:
		\begin{align*}
			\text{Res}_F(f) := \left\vert \frac{d\mu}{\kappa_1 \wedge ... \wedge \kappa_n} \right \vert \text{Res}_{\kappa_n=0} \dots \text{Res}_{\kappa_1=0} (f).
		\end{align*}
	\end{definition}
	Finally we can define the Jeffrey-Kirwan residue operation:
	\begin{definition}\label{definitionJKresidue}
		Fix a finite projective set $\mathfrak{A}\subset \Gamma$ and a sum-regular stability $\xi \in \mathfrak{a}^\vee$. 
		The Jeffrey-Kirwan residue of a meromorphic function $f : \mathfrak{a}\otimes_{\mathbb{R}}\mathbb{C} \dashrightarrow \mathbb{C}$ is
		\begin{align*}
			\text{JK}^{\mathfrak{A}}(f, \xi) := \sum_{F \in \mathcal{F}(\mathfrak{A}, \xi)} \text{Res}_F(f).
		\end{align*}
	\end{definition}
	Given $a \in \mathfrak{a}\otimes_\mathbb{R}\mathbb{C}$, we will denote the residue at this point with
	\begin{align*}
		\text{JK}^\mathfrak{A}_a(f, \xi) := \text{JK}^\mathfrak{A}(f \circ \tau_a, \xi),
	\end{align*}
	where $\tau_a$ is translation by $a$.
	In \cite{OntaniStoppa} it is shown how the residue map $\text{JK}^\mathfrak{A}$ behaves with respect to rescaling of the coordinates on $\mathfrak{a}$:
	\begin{lem}\label{rescalingJK}
		For every function $F$ on $\mathfrak{a}\otimes_\mathbb{R} \mathbb{C}$, under rescaling of the variables by a constant $\lambda \in \mathbb{C}^\ast$, the JK residues rescales as
		\begin{align*}
			\text{JK}^{\mathfrak{A}} (F(\lambda u)) = \lambda^{-\text{dim}T} \text{JK}^{\mathfrak{A}} (F(u)).
		\end{align*}
	\end{lem}
	\subsection{Abelian Jeffrey-Kirwan localization.}
	Roughly speaking, this is a formula that allows to integrate cohomology classes on quotients of varieties by tori in terms of the data at the fixed loci of such actions. 
	
	Let $T$ be a torus acting effectively on a smooth variety $V$. Pick a regular linearization $\mathcal{L}$ (or equivalently $\xi \in \mathfrak{t}_\mathbb{Z}^\vee$) so that the quotient $V/\!/T$ is projective. Notice that since the action on the semistable locus is locally free, the projection $p : V^{T\text{-sst}}\rightarrow V/\!/T$, where the $T$-action on the target is assumed to be trivial, induces an morphism of equivariant Chow groups
	\begin{align*}
		A^\ast(V/\!/T)[t] \xrightarrow{p^\ast} A^\ast_T(V^{T\text{-sst}})
	\end{align*}
	which becomes an isomorphism once we restrict it to $A^\ast(V/\!/T)$ and we work with $\mathbb{Q}$-coefficients (see \cite{EdidinEIT}, theorem 4). 
	The main ingredient we need in this section is the \textit{Kirwan map}, which is the composition
	\begin{align*}
		r : A^\ast_T(V) \xrightarrow{i^\ast} A^\ast_T(V^{T\text{-sst}}) \xrightarrow{(p^\ast)^{-1}} A^\ast(V/\!/T).
	\end{align*}
	The same picture appears in K-theory, so we find the analogous \textit{descent map}
	\begin{align*}
		d : K_T^0(V) \xrightarrow{i^\ast} K_T^0(V^{T\text{-sst}}) \xrightarrow{(p^\ast)^{-1}} K^0(V/\!/T).
	\end{align*}
	Notice that if $E$ is an equivariant bundle on $V$, then $d(E)$ is the class of the bundle induced on the quotient. 
	\begin{rem}
		The morphism $r$ preserves the grading, and can readily be extended to a morphism $\hat{A}^\ast_T(V):= \prod_{k=0}^\infty A^k_T(V) \rightarrow A^\ast(V/\!/T)$ by working in each degree.
	\end{rem}
	Let $V$ be a complex linear space,  $\mathcal{W}$ be the set of weights of the $T$-representation and $\xi \in \mathfrak{t}^\vee_\mathbb{Z}$ be the regular stability corresponding to the linearization. We are in the context of Example \ref{AbstractJKExample}, hence this data defines a JK residue
	\begin{align*}
		\text{JK}^{\mathcal{W}}(-, \tilde{\xi})
	\end{align*}
	for every sum-regular perturbation $\tilde{\xi}$ of $\xi$.
	In this context the equivariant cohomology of $V$ is $A^\ast_T(V) \simeq \text{Sym}(\mathfrak{t}^\vee)$ and the following version of the Jeffrey-Kirwan localization formula holds as proven in \cite{SzenesVergne}:
	\begin{theorem}\label{SVlocalization}
		Let a torus $T$ act on a linear space $V$ with a regular stability $\xi \in \mathfrak{t}_\mathbb{Z}^\vee$ so that the quotient $V/\!/T$ is projective.	Let $\tilde{\xi} \in \mathfrak{t}_{\mathbb{R}}^\vee$ be any sum-regular perturbation of $\xi$. For every equivariant class $\alpha \in A^\ast_T(V)$, we have
		\begin{align*}
			\int_{V/\!/T} r(\alpha) = \text{JK}^{\mathcal{W}}\left(\frac{\alpha}{e^T(T_V)}, \tilde{\xi}\right).
		\end{align*}
	\end{theorem}
	\begin{rem}
		Notice that since the quotient is projective, $V$ has trivial fixed part $V^T = 0$, hence its equivariant Euler class is invertible in the localized equivariant cohomology.
	\end{rem}
	Now we are going to derive a simple corollary of this fact. We start with the following
	\begin{lem}\label{KirwanDiagrams}
		The diagrams
		\[
		\begin{tikzcd}
			K_T^0(V) \arrow[r, "d"] \arrow[d, "c_\bullet^T"] & K^0(V/\!/T) \arrow[d, "c_\bullet"]\\
			\hat{A}_T^\ast(V) \arrow[r, "r"] & A^\ast(V/\!/T)
		\end{tikzcd} \quad \text{and} \quad
		\begin{tikzcd}
			K_T^0(V) \arrow[r, "d"] \arrow[d, "\text{ch}^T(-)\text{td}^T(T_V)"] & K^0(V/\!/T) \arrow[d, "\text{ch}(-)\text{td}(T_{V/\!/T})"]\\
			\hat{A}_T^\ast(V) \arrow[r, "r"] & A^\ast(V/\!/T)
		\end{tikzcd}
		\]
		are commutative.
	\end{lem}
	\begin{proof}
		Since $c_\bullet$ and $\text{ch}$ are natural transformations between $K$ and $\hat{A}$, we have that the diagram
		\[
		\begin{tikzcd}
			K^0(V/\!/T) \arrow[r, "p^\ast"] \arrow[d, "c_\bullet"] & K_T^0(V^{T\text{-sst}}) \arrow[d, "c_\bullet^T"] & K_T^0(V) \arrow[l, swap, "i^\ast"] \arrow[d, "c_\bullet^T"]\\
			A^\ast(V/\!/T) \arrow[r, "p^\ast"] & \hat{A}_T^\ast(V^{T\text{-sst}}) & \hat{A}_T^\ast(V) \arrow[l, swap, "i^\ast"]
		\end{tikzcd}
		\]
		is commutative, and the same happens for the Chern character instead of $c_\bullet$. The desired diagrams follow by inverting $p^\ast$ in both $K$-theory and Chow cohomology. 
	\end{proof}
	We now want to derive the precise result we use in this paper. First, let $P$ be a polynomial and let $E_1, ..., E_N$ be T-equivariant vector bundles on $V$. The diagrams above show that for $k \in \mathbb{N}^N$
	\begin{align*}
		\int_{V/\!/T} P(c_{k_1}(d(E_1)), ..., c_{k_N}(d(E_N))) = \text{JK}^\mathcal{W}\left(\frac{P(c^T_{k_1}(E_1), ..., c^T_{k_N}(E_N))}{e^T(T_V)}, \tilde{\xi}\right).
	\end{align*}
	Now assume that we also have a trivial $\mathbb{C}^\ast$-action on $V$ descending to the one on $V/\!/T$. The diagram of descent maps 
	\[\begin{tikzcd}
		K_{T \times \mathbb{C}^\ast}^0(V) \arrow[r, "d"] \arrow[d, "\sim"] & K_{\mathbb{C}^\ast}^0(V/\!/T) \arrow[d, "\sim"]\\
		K_{T}^0(V)[y]_y \arrow[r, "d"] & K^0(V/\!/T)[y]_y
	\end{tikzcd}\]
	commutes, where the lowest horizontal arrow is the descent map acting on the coefficients of the Laurent polynomial ring. This shows that for every $N$-tuple of $T \times \mathbb{C}^\ast$-equivariant bundles $E_1, ...,E_N$ on $V$ and a polynomial $P$ we have
	\begin{align*}
		\int_{V/\!/T} P(c^{\mathbb{C}^\ast}_{k_1}(d(E_1)), ..., c^{\mathbb{C}^\ast}_{k_N}(d(E_N))) = \text{JK}^\mathcal{W}\left(\frac{P(c^{T\times \mathbb{C}^\ast}_{k_1}(E_1), ..., c^{T\times \mathbb{C}^\ast}_{k_N}(E_N))}{e^T(T_V)}, \tilde{\xi}\right).
	\end{align*}
	Let $F$ be a $T\times \mathbb{C}^\ast$-equivariant vector bundle on $V$ with trivial $\mathbb{C}^\ast$-fixed part. Then $e^{\mathbb{C}^\ast}(d(F))^{-1}$ exists in $A^\ast_{\mathbb{C}^\ast}(V/\!/T)_s$ and can be written as a power series in the equivariant variable $s$ whose coefficients are polynomials in the Chern classes of $d(F)$. Hence, for every other equivariant bundle $E$ on $V$:
	\begin{align*}
		\int_{V/\!/T} \frac{e^{\mathbb{C}^\ast}(d(E))}{e^{\mathbb{C}^\ast}(d(F))} = \text{JK}^\mathcal{W}\left(\frac{e^{T \times \mathbb{C}^\ast}(E)}{e^{T \times \mathbb{C}^\ast}(F) \,e^T(T_V)}, \tilde{\xi}\right).
	\end{align*}
	Analogously, we can use the equivariant Hirzebruch-Riemann-Roch theorem of \cite{EdidinRR} to show the second part of the following
	\begin{cor}\label{DirectJK}
		Let a torus $T$ act on a linear space $V$ with a regular stability $\xi \in \mathfrak{t}_\mathbb{Z}^\vee$ so that the quotient $V/\!/T$ is projective and let $\tilde{\xi}$ be a sum-regular perturbation of $\xi$.
		Assume that $\mathbb{C}^\ast$ acts trivially on $V$, inducing a trivial action on the quotient $V/\!/T$. For every $E \in K_{T\times {\mathbb{C}^\ast}}^0(V)$ admitting an Euler class
		\begin{align*}
			\int_{V/\!/T} e^{\mathbb{C}^\ast}(d(E)) = \text{JK}^{\mathcal{W}}\left(\frac{e^{T\times {\mathbb{C}^\ast}}(E)}{e^T(T_V)}, \tilde{\xi}\right).
		\end{align*}
		For every $F \in K_{T\times {\mathbb{C}^\ast}}^0(V)$
		\begin{align*}
			\text{ch}^{\mathbb{C}^\ast}\left(\chi(V/\!/T, d(F))\right) = \text{JK}^{\mathcal{W}}\left(\text{ch}^{T\times {\mathbb{C}^\ast}}\left( \frac{F}{\Lambda_{-1} \Omega_V}\right), \tilde{\xi}\right).
		\end{align*}
	\end{cor}
	\subsection{Some computations with descent maps.}\label{translations}
	Consider a linear space $V$ with a $T \times {\mathbb{C}^\ast}$-action so that $V/\!/T$ is fixed by ${\mathbb{C}^\ast}$. 
	Consider the ${\mathbb{C}^\ast}$-equivariant descent map for this action
	\begin{align*}
		\tilde{d} : K_{T \times {\mathbb{C}^\ast}}^0(V) \rightarrow K^0_{\mathbb{C}^\ast}(V).
	\end{align*}
	We can also consider the equivariant descent map 
	\begin{align*}
		d : K_{T\times {\mathbb{C}^\ast}}^0(V) \rightarrow K^0_{\mathbb{C}^\ast}(V/\!/T)
	\end{align*}
	for the trivial action on $V$ introduced in the previous section. Here we describe the relation between these two maps, which becomes clearer when we split the torus $T \simeq (\mathbb{C}^\ast)^n$ so that $K^0_{T}(V) \simeq \mathbb{Q}[t_1^{\pm 1}, ..., t_n^{\pm 1}]$.
	\begin{lem}\label{descentDiagram}
		Assume that there is group homomorphism $\phi: {\mathbb{C}^\ast} \rightarrow T$ so that $s \cdot v = \phi(s) \cdot v$ for every $s \in {\mathbb{C}^\ast}$ and $v \in V$. Set $a_1, ..., a_n \in \mathbb{Z}$ so that $\phi(s)_i = s^{a_i}$ for all $i$. Then the descent maps $\tilde{d}$ and $d$ are related by the following commutative diagram:
		\[\begin{tikzcd}
			\mathbb{Q}[t_1^{\pm 1}, ..., t_n^{\pm 1}, y^{\pm 1}] \arrow[rr, "\tau"] \arrow[dr, swap, "\tilde{d}"] & & \mathbb{Q}[t_1^{\pm 1}, ..., t_n^{\pm 1}, y^{\pm 1}] \arrow[ld, "d"]\\
			& \mathbb{Q}[y^{\pm 1}] &
		\end{tikzcd}\]
		where $\tau(y) = y$ and $\tau(t_i):= t_i y^{-{a_i}}$ for all $i$.
	\end{lem}
	\begin{proof}
		By definition both $\tilde{d}$ and $d$ are morphisms of $\mathbb{Q}[y]_y$-algebras, hence they are determined by the images of the $t_i$. Now notice that $\tilde{d}(t_i)$ is, at the non-equivariant level, the same bundle as $d(t_i)$. We can study its ${\mathbb{C}^\ast}$-equivariant structure on the fiber above a chosen point $[v] \in V/\!/T$:
		\begin{align*}
			s \cdot [v, w] = [s \cdot v, w] = [\phi(s) \cdot v, w] = [v, \phi(s)^{-1} \cdot w] = [v, s^{-{a_i}} w],
		\end{align*}
		concluding the proof.
	\end{proof}
	Now notice that the diagram
	\[\begin{tikzcd}
		K_{T\times {\mathbb{C}^\ast}}^0(V) \arrow[r, "\tau"]\arrow[d, "\text{ch}^{T \times {\mathbb{C}^\ast}}"] & K_{T\times {\mathbb{C}^\ast}}^0(V) \arrow[d, "\text{ch}^{T \times {\mathbb{C}^\ast}}"]\\
		\hat{A}^\ast_{T\times {\mathbb{C}^\ast}}(V) \arrow[r, "\tau"]& \hat{A}^\ast_{T\times {\mathbb{C}^\ast}}(V)
	\end{tikzcd}\]
	commutes once we set $u_i := e^{\mathbb{C}^\ast}(t_i)$ and we define the lower $\tau$ to be the translation $s \mapsto s$, $u_i \mapsto u_i - a_i s$. The same diagram, with the Chern class in place of the Chern character, commutes. We can combine this with corollary \ref{DirectJK} to obtain the following
	\begin{cor}\label{superDirectJK}
		Let a torus $T$ act on a linear space $V$ with a regular stability $\xi \in \mathfrak{t}_\mathbb{Z}^\vee$ so that the quotient $V/\!/T$ is projective and let $\tilde{\xi}$ be a sum-regular perturbation of $\xi$. Let $\mathbb{C}^\ast$ act on $V$ via a morphism $\mathbb{C}^\ast \rightarrow T$ corresponding to an element $a \in \mathfrak{t}$. If $E \in K_{T\times \mathbb{C}^\ast}^0(V)$ admits an Euler class
		\begin{align*}
			\int_{V/\!/T} e^{\mathbb{C}^\ast}(\tilde{d}(E)) = \text{JK}_{-a s}^{\mathcal{W}}\left(\frac{e^{T\times {\mathbb{C}^\ast}}(E)}{e^T(T_V)}, \tilde{\xi}\right).
		\end{align*}
		For every $F \in K_{T\times {\mathbb{C}^\ast}}^0(V)$
		\begin{align*}
			\text{ch}^{\mathbb{C}^\ast}\left(\chi(V/\!/T, \tilde{d}(F))\right) = \text{JK}_{-a s}^{\mathcal{W}}\left(\text{ch}^{T\times {\mathbb{C}^\ast}}\left( \frac{F}{\Lambda_{-1} \Omega_V}\right), \tilde{\xi}\right).
		\end{align*}
	\end{cor}
	\subsection{The equivariant Martin's formula.}
	The last step towards a localization formula for quotients by nonabelian groups is to find a way to lift integrals to the quotient by the maximal subtorus.
	The formula of Martin, which we now recall, is a useful tool to do so. Let $V$ be a smooth variety together with the action of a reductive algebraic group $G$ and a linearization $\mathcal{L}$. Let $\mathcal{W}, \mathfrak{R} \subset \mathfrak{t}_\mathbb{Z}^\vee$ be the weights for the action and the roots of $G$ respectively. Let $W$ be the Weyl group of $G$.
	\begin{definition}
		The \textit{roots bundle} of $V/\!/T$ is the vector bundle $\mathcal{R}$ induced by the trivial vector bundle $\mathfrak{g}/\mathfrak{t}$ on $V$ with $T$-equivariant structure given by the adjoint action.
	\end{definition}
	\begin{rem}
		Since the roots of the complex Lie algebra $\mathfrak{g}$ come in positive/negative pairs, there is an isomorphism of $T$-representations $\mathfrak{g}/\mathfrak{t} \simeq (\mathfrak{g}/\mathfrak{t})^\vee$, hence the roots bundle is self-dual: $\mathcal{R} \simeq \mathcal{R}^\vee$.
	\end{rem}
	Martin's formula, proven in \cite{Martin}, is stated in terms of the projection $\pi : (V/\!/T)^{G\text{-sst}} \rightarrow V/\!/G$	and the inclusion $j : (V/\!/T)^{G\text{-sst}} \hookrightarrow V/\!/T$.
	\begin{pro}
		Assume that $V/\!/T$ and $V/\!/G$ are projective varieties. Let $\alpha \in A^\ast(V/\!/G)$ and $\beta \in A^\ast(V/\!/T)$ be such that $\pi^\ast \alpha = j^\ast \beta$. Then
		\begin{align*}
			\int_{V/\!/G}\alpha = \frac{1}{\vert W \vert} \int_{V/\!/T} \beta \cdot e(\mathcal{R}).
		\end{align*}
	\end{pro}
	It's worth recalling that this, together with theorem \ref{SVlocalization}, yields the Jeffrey-Kirwan localization formula for linear spaces:
	\begin{theorem}
		Let $G \curvearrowright V$ be a $G$-representation and $\xi$ be a regular stability so that the quotients by $G$ and $T$ are projective. Let $\tilde{\xi}$ be a sum-regular perturbation of $\xi$. Then, for every $\omega \in A_G^\ast(V)$, we have
		\begin{align*}
			\int_{V/\!/G} r(\omega) = \frac{1}{\vert W \vert}\text{JK}^\mathcal{W} \left(\frac{\omega \cdot \prod_{\alpha \in \mathfrak{R}} \alpha}{e^T(T_V)}, \tilde{\xi} \right)
		\end{align*}
	\end{theorem}
	We now want to discuss an extension of this formula to the equivariant setting.
	Following the lines of the proof of Martin's formula in \cite{Martin}, it's clear that he proves the following statement:
	\begin{pro}
		Let $p, q$ be proper maps to a variety $X$ so that the following diagram
		\[\begin{tikzcd}
			V/\!/T \arrow[rr, dashed, "\pi"] \arrow[dr,swap, "p"] & & V/\!/G \arrow[dl, "q"]\\
			& X &
		\end{tikzcd}\]
		is commutative. Then, for every $\alpha \in A^\ast(V/\!/G)$ and $\beta \in A^\ast(V/\!/T)$ such that $\pi^\ast \alpha = j^\ast \beta$, the formula
		\begin{align*}
			q_\ast \alpha = \frac{1}{\vert W \vert} p_\ast \left(\beta \cdot e(\mathcal{R})\right)
		\end{align*}
		holds true in $A^\ast(X)$.
	\end{pro}
	Let ${\mathbb{C}^\ast}$ be another reductive algebraic group acting on $V$ such that the action commutes with the one of $G$ and that $\mathcal{L}$ is ${\mathbb{C}^\ast}$-equivariant. This induces an ${\mathbb{C}^\ast}$-action on the quotients $V/\!/T$ and $V/\!/G$.
	\begin{rem}
		The roots bundle has a canonical ${\mathbb{C}^\ast}$-equivariant structure induced by the trivial ${\mathbb{C}^\ast}$-action on the fiber $\mathfrak{g}/\mathfrak{t}$. Notice that this structure is not necessarily trivial.
	\end{rem}
	\begin{pro}\label{equivariantMartinFormula}
		Let $\alpha \in A_{\mathbb{C}^\ast}^\ast(V/\!/G)$ and $\beta \in A_{\mathbb{C}^\ast}^\ast(V/\!/T)$ be such that $\pi^\ast \alpha = j^\ast \beta$. Then
		\begin{align*}
			\int_{V/\!/G}\alpha = \frac{1}{\vert W \vert} \int_{V/\!/T} \beta \cdot e^{\mathbb{C}^\ast}(\mathcal{R})
		\end{align*}
	\end{pro}
	\begin{proof}
		Clearly we can prove this assuming that $\alpha$ is homogeneous of some degree $i$. The construction of corollary 2 \cite{EdidinEIT} of the equivariant operational Chow groups in terms of finite dimensional approximations of $E{\mathbb{C}^\ast}$ allows to say that there is a smooth ${\mathbb{C}^\ast}$-variety $U$ so that $A_{\mathbb{C}^\ast}^l(Y) \simeq A^l(Y_{\mathbb{C}^\ast})$ for all $l\leq i$, where for a ${\mathbb{C}^\ast}$-variety $Y$ we write $Y_{\mathbb{C}^\ast} := \frac{Y \times U}{{\mathbb{C}^\ast}}$. If we put the trivial $G$-action on $U$ and we pull back $\mathcal{L}$ onto $V \times U$ we find that
		\begin{align*}
			V_{\mathbb{C}^\ast}/\!/G \simeq (V/\!/G)_{\mathbb{C}^\ast} \quad \text{ and }\quad V_{\mathbb{C}^\ast}/\!/T \simeq (V/\!/T)_{\mathbb{C}^\ast}.
		\end{align*}
		The result follows from applying Martin's formula to the diagram
		\[\begin{tikzcd}
			(V/\!/T)_{\mathbb{C}^\ast} \arrow[rr, dashed, "\pi"] \arrow[dr,swap, "p_2"] & & (V/\!/G)_{\mathbb{C}^\ast} \arrow[dl, "p_2"]\\
			& B{\mathbb{C}^\ast} &
		\end{tikzcd}\]
		once we recall that the pushforward along the diagonal lines is the equivariant integration.
	\end{proof}
	In particular we have the following:
	\begin{cor}\label{zeroIntegral}
		Assume that $V/\!/G = \emptyset$. Then
		\begin{align*}
			\int_{V/\!/T} \beta \cdot e^{\mathbb{C}^\ast}(\mathcal{R}) = 0
		\end{align*}
		for every $\beta \in A^\ast_{\mathbb{C}^\ast}(V/\!/T)$.
	\end{cor}
	\section{Virtual invariants.}\label{sectionVir}
	In this section we recall the definition of the invariants we want to study and we compute them using the results of the previous sections.
	\subsection{Generalities about virtual invariants.}
	Consider a scheme $X$ with an action of $\mathbb{C}^\ast$ and an equivariant perfect obstruction theory $\mathbb{E}^\bullet \rightarrow \mathbb{L}_X$. The $K$-theory class of $\mathbb{E}^\bullet$ is called the \textit{virtual cotangent bundle} of $X$, also denoted with $\Omega_X^{\text{vir}} \in K_{\mathbb{C}^\ast}^0(X)$. The dual complex $\mathbb{E}_\bullet$ is called \textit{virtual tangent bundle} of $X$, and its $K$-theory class is denoted with $T^{\text{vir}}_X$. There are two important invariants of $X$, which we now describe. Let $C \subset \mathbb{E}_1 := (\mathbb{E}^{-1})^\vee$ be the cone over $X$ built, as described in \cite{Behrend}, using the intrinsic normal cone of $X$ and consider the zero section $i : X \hookrightarrow \mathbb{E}_1$. The first $K$-theoretic invariant we encounter is the \textit{virtual structure sheaf} of $X$, i.e. the class in $K^{{\mathbb{C}^\ast}}_0(X)$ given by $\mathcal{O}_X^{\text{vir}} := \textbf{L}i^\ast \mathcal{O}_C$. 
	The basic cohomological invariant of $X$ is its \textit{virtual fundamental class}, defined as the refined intersection $[X]^\text{vir} := i^![C] \in A^{{\mathbb{C}^\ast}}_{\text{vd}(X)}(X)$, where $\text{vd}(X)$ is the \textit{virtual dimension} of $X$ which can be computed as
	\begin{align*}
		\text{vd}(X) = \text{rk}\mathbb{E}^0 - \text{rk}\mathbb{E}^{-1}
	\end{align*}
	whenever we can write $\mathbb{E}^\bullet = [\mathbb{E}^{-1} \rightarrow \mathbb{E}^0]$.
	If $X$ is proper, we can use the pushforwards to a point (denoted with $\chi$ in $K$-theory and $\int$ in cohomology) to extract numbers out of these invariants: we set
	\begin{align*}
		\chi^{\text{vir}}(X, -) : K^{\mathbb{C}^\ast}_0(X) \rightarrow \mathbb{Z}[y]_y \quad &: \quad \chi^{\text{vir}}(X,V) := \chi(X, V \otimes \mathcal{O}^{\text{vir}_X}),\\
		\int_{[X]^{\text{vir}}} : A^\ast_{\mathbb{C}^\ast}(X) \rightarrow \mathbb{Z}[s] \quad &: \quad \int_{[X]^{\text{vir}}} \alpha := \int_X \alpha \cap [X]^{\text{vir}}.
	\end{align*}
	These two operations are related by the virtual Hirzebruch-Riemann-Roch theorem
	\begin{align}\label{virtualHRR}
		\text{ch}^{\mathbb{C}^\ast}\left(\chi^{\text{vir}}(X,V)\right) = \int_{[X]^{\text{vir}}} \text{ch}^{\mathbb{C}^\ast}(V) \text{td}^{\mathbb{C}^\ast}(T^{\text{vir}}_X).
	\end{align}
	for every $V \in K^0_{\mathbb{C}^\ast}(X)$, see \cite{EdidinRR, FantechiGottsche}.
	\begin{rem}
		The Chern character on the left is the equivariant Chern character $\text{ch}^{\mathbb{C}^\ast} : K^0_{\mathbb{C}^\ast}(\text{pt}) \rightarrow \hat{A}_{\mathbb{C}^\ast}^\ast(\text{pt})$, where $\hat{A}_{\mathbb{C}^\ast}^\ast(Y)$ is defined as
		\begin{align*}
			\hat{A}_{\mathbb{C}^\ast}^\ast(Y) := \prod_{k=0}^\infty A^k_{\mathbb{C}^\ast}(Y)
		\end{align*}
		for every ${\mathbb{C}^\ast}$-scheme $Y$. Since we denote with $y$ the K-theory class of $\mathfrak{s}$ and with $s$ its first Chern class, we have that this character is just the evaluation
		\begin{align*}
			\text{ch}^{\mathbb{C}^\ast} : \mathbb{Q}[y]_y \rightarrow \mathbb{Q}[\![s]\!] \quad : \quad y \mapsto e^s.
		\end{align*}
	\end{rem}
	An important fact shown in \cite{Graber} is that, over the fixed locus $X^{\mathbb{C}^\ast}$, the complex $\mathbb{E}^\bullet$ splits into a moving and a fixed part $\mathbb{E}^\bullet_f$ and $\mathbb{E}^\bullet_m$ and the fixed part induces a perfect obstruction theory on $X^{\mathbb{C}^\ast}$. We will denote with $\Omega_{X^{\mathbb{C}^\ast}}^{\text{vir}}$ the $K$-theory class of $\mathbb{E}^\bullet_f$ and with $(\mathcal{N}^{\text{vir}})^\vee$ the class of $\mathbb{E}^\bullet_m$. They are the virtual cotangent bundle and \textit{virtual conormal bundle} of $X^{\mathbb{C}^\ast}$ in $X$.
	The virtual localization formula of \cite{Graber} shows that, if $X$ is proper, then for every $V \in K_{\mathbb{C}^\ast}^0(X)$ and $\alpha \in A^\ast_{\mathbb{C}^\ast}(X)$:
	\begin{align*}
		\chi^{\text{vir}}(X, V) = \chi^{\text{vir}}\left(X^{\mathbb{C}^\ast}, \frac{V}{\Lambda_{-1}(\mathcal{N}^{\text{vir}})^\vee}\right), \qquad \int_{[X]^{\text{vir}}} \alpha = \int_{[X^{\mathbb{C}^\ast}]^{\text{vir}}} \frac{\alpha}{e^{\mathbb{C}^\ast}(\mathcal{N}^{\text{vir}})}.
	\end{align*}
	Notice that the right hand sides are well defined even if $X$ is not proper: they just require properness of the fixed subscheme $X^{\mathbb{C}^\ast}$. 
	\begin{align*}
		\underline{\text{In case } X \text{ is not proper but } X^{\mathbb{C}^\ast} \text{ is, we will treat these formulae as definitions}.}
	\end{align*}
	Since in order to invert the relevant classes we have to extend the K and Chow modules, in case $X$ is not proper the localized invariants will lie in the extended $K$ and Chow modules of the point:
	\begin{align*}
		\chi^{\text{vir}}(X, V) \in \mathbb{Q}[\![y]\!]_y, \qquad \int_{[X]^{\text{vir}}} \alpha \in \mathbb{Q}[s]_s.
	\end{align*}
	\subsection{A cohomological invariant.}
	In the case we will study, the obstruction theory has virtual dimension zero:
	\begin{align*}
		\text{vd}(X)=0,
	\end{align*}
	hence there is a cohomological invariant we are interested in:
	\begin{definition}\label{DT}
		The \textit{(equivariant) DT invariant} of $X$ is the rational number given by the degree of the virtual fundamental class
		\begin{align*}
			\text{DT} := \int_{[X]^{\text{vir}}} {1}\in \mathbb{Q}.
		\end{align*}
	\end{definition}
	\begin{rem}\label{constantDT}
		Notice that if $X$ is proper then $\text{DT}$ is an element of $A_0^{\mathbb{C}^\ast}(\text{pt}) \simeq \mathbb{Z}$ by degree reasons. In case of a nonproper $X$, the DT invariant should in principle belong to $\mathbb{Q}[s]_s$. On the other hand
		\begin{align*}
			\text{rk}(\mathcal{N}^{\text{vir}}) = \text{vd}(X) - \text{vd}(X^{\mathbb{C}^\ast}) = - \text{vd}(X^{\mathbb{C}^\ast}).
		\end{align*}
		Consider the localized graded ring $(A^\ast_{{\mathbb{C}^\ast}}(X^{\mathbb{C}^\ast}))_s$. The induced decomposition in graded pieces has degree $k$ component $C^k$ given by
		\begin{align}\label{gradedDecomposition}
			C^k \simeq \bigoplus_{
				\substack{p,q \in \mathbb{Z}\\p+q=k}} s^p \cdot  A^q(X^{\mathbb{C}^\ast}).
		\end{align}
		By what we have written above ${e^{\mathbb{C}^\ast}(\mathcal{N}^{\text{vir}})}^{-1} \in C^{\text{vd}(X^{\mathbb{C}^\ast})}$. Since $[X^{\mathbb{C}^\ast}]^{\text{vir}}$ belongs to $A_{\text{vd}(X^{\mathbb{C}^\ast})}(X^{\mathbb{C}^\ast})$, their pairing is indeed in $A_0^{\mathbb{C}^\ast}(\text{pt}) \simeq \mathbb{Q}$.
	\end{rem}
	\subsection{K-theoretic invariants.}
	There are two more invariants we want to consider.
	The the virtual canonical bundle of $X$ is the class 
	\begin{align*}
		K^{\text{vir}}_X := \text{det}(\Omega_X^{\text{vir}}) \in K_{{\mathbb{C}^\ast}}^0(X)
	\end{align*} 
	and it admits a canonical square root (as shown in \cite{OhThomas}) once we work with rational coefficients and we add the square root of $y$:
	\begin{align*}
		\sqrt{K^{\text{vir}}_X} \in K_{{\mathbb{C}^\ast}}^0(X) \otimes \mathbb{Q}[y^{1/2}].
	\end{align*}
	\begin{definition}\label{Chiy}
		The \textit{(equivariant) virtual Hirzebruch genus} of $X$ is the Laurent power series
		\begin{align*}
			\chi_y := \chi^{\text{vir}}\left(X, \sqrt{K_X^{\text{vir}}}\right) \in \mathbb{Q}[\![y^{1/2}]\!]_{y^{1/2}}.
		\end{align*}
	\end{definition}
	Another interesting invariant is expressed in terms of the following equivariant $K$-theory class defined in \cite{Ricolfi}: for every $V \in K_{\mathbb{C}^\ast}^0(X)$ set
	\begin{align*}
		\mathcal{E}_{1/2}(V) := \bigotimes_{n \geq 1} \text{Sym}_{q^n}\left(V \oplus V^\vee \right) \in K^0_{{\mathbb{C}^\ast}}(X)[\![q]\!].
	\end{align*}
	having constant term in $q$ equal to 1. Notice that $\mathcal{E}_{1/2}$ defines a group homomorphism between $(K_{\mathbb{C}^\ast}^0(X), +)$ and $(1+q\cdot K_{\mathbb{C}^\ast}^0(X)[\![q]\!], \otimes)$.
	\begin{definition}\label{Ell}
		The \textit{(equivariant) virtual chiral elliptic genus} of $X$ is the Euler characteristic
		\begin{align*}
			\text{Ell}_{1/2}(q,y) := \chi^{\text{vir}}\left(X, \mathcal{E}_{1/2}(T^{\text{vir}}_X) \otimes \sqrt{K^{\text{vir}}_X}\right)
		\end{align*}
		which belongs to $\mathbb{Q}[\![q, y^{1/2}]\!]_{y^{1/2}}$.
	\end{definition}
	The three invariants $\text{DT}$, $\chi_y$ and $\text{Ell}(q, y)$ are related in a simple way:
	\begin{pro}\label{convergence}
		The following relations hold true for $X$ with an equivariant perfect obstruction theory of virtual dimension zero:
		\begin{itemize}
			\item $\chi_y = \text{Ell}(0, y)$.
			\item $\text{DT} = \chi_1$.
		\end{itemize}
	\end{pro}
	\begin{proof}
		The first statement follows immediately from the fact that $\mathcal{E}_{1/2}(T_X^{\text{vir}})$ belongs to $1+ q \cdot K_{\mathbb{C}^\ast}^0(X)[\![q]\!]$. By the virtual Hirzebruch-Riemann-Roch theorem we can write
		\begin{align*}
			\text{ch}^{\mathbb{C}^\ast}(\chi_{y}) = \int_{[X^{\mathbb{C}^\ast}]^\text{vir}} \frac{\text{ch}^{\mathbb{C}^\ast}\left(\sqrt{K_X^{\text{vir}}}\right)
				\text{td}^{\mathbb{C}^\ast}\left(T_X^{\text{vir}}\right)}{e^{\mathbb{C}^\ast}\left(\mathcal{N}^{\text{vir}}\right)}.
		\end{align*}
		By definition we have that
		\begin{align*}
			\text{ch}^{\mathbb{C}^\ast}\left(\sqrt{K_X^{\text{vir}}}\right)
			\text{td}^{\mathbb{C}^\ast}\left(T_X^{\text{vir}}\right) \in \hat{A}^\ast_{\mathbb{C}^\ast}(X^{\mathbb{C}^\ast}) \simeq \prod_{k=0}^\infty A^k_{\mathbb{C}^\ast}(X^{\mathbb{C}^\ast}).
		\end{align*}
		and that the degree $k=0$ part is 1.
		Recall that in remark \ref{constantDT} we have shown that $e^{\mathbb{C}^\ast}(\mathcal{N}^{\text{vir}})$ belongs to $C^{\text{vd}(X^{\mathbb{C}^\ast})}$, where $C^k$ is the degree $k$ component of $(A^\ast_{\mathbb{C}^\ast}(X^{\mathbb{C}^\ast}))_s$ as described in (\ref{gradedDecomposition}).
		This shows that the localized class satisfies
		\begin{align*}
			\frac{\text{ch}^{\mathbb{C}^\ast}\left(\sqrt{K_X^{\text{vir}}}\right)
				\text{td}^{\mathbb{C}^\ast}\left(T_X^{\text{vir}}\right)}{e^{\mathbb{C}^\ast}\left(\mathcal{N}^{\text{vir}}\right)} \in \prod_{k=\text{vd}(X^{\mathbb{C}^\ast})}^\infty C^k
		\end{align*}
		and that the degree $k= \text{vd}(X^{\mathbb{C}^\ast})$ is $e^{\mathbb{C}^\ast}(\mathcal{N}^{\text{vir}})^{-1}$,
		therefore the pairing with $[X^{\mathbb{C}^\ast}]^{\text{vir}}$ satisfies
		\begin{align*}
			\chi_{e^s} \in \prod_{k=\text{vd}(X^{\mathbb{C}^\ast})}^\infty A^{\mathbb{C}^\ast}_{\text{vd}(X^{\mathbb{C}^\ast})-k}(\text{pt}) \simeq \prod_{k=\text{vd}(X^{\mathbb{C}^\ast})}^\infty \mathbb{Q} \cdot s^{k-\text{vd}(X^{\mathbb{C}^\ast})},
		\end{align*}
		and the constant part for $k=\text{vd}(X^{\mathbb{C}^\ast})$ is 
		\begin{align*}
			\int_{[X^{\mathbb{C}^\ast}]^{\text{vir}}} \frac{1}{e^{\mathbb{C}^\ast}(\mathcal{N}^{\text{vir}})} = DT.
		\end{align*}
	\end{proof}
	At some point we will need the following variant of the morphism $\mathcal{E}_{1/2}$: given an equivariant K-theory class $V$ so that $\Lambda_{-1}V^\vee$ is invertible, consider the class
	\begin{align}\label{tildeEclass}
		\hat{\mathcal{E}}_{1/2}(V) := \frac{\mathcal{E}_{1/2}(V)\otimes \sqrt{\text{det}(V^\vee)}}{\Lambda_{-1}V^\vee}
	\end{align}
	Notice that since all the classes in its definition are multiplicative, $\hat{\mathcal{E}}_{1/2}$ sends sums into products:
	\begin{align*}
		\hat{\mathcal{E}}_{1/2}(V+W) = \hat{\mathcal{E}}_{1/2}(V)\otimes\hat{\mathcal{E}}_{1/2}(W).
	\end{align*}
	\subsection{Invariants of critical loci.}
	The invariants we have introduced are generally hard to compute. Part of the difficulty lies in the fact that $X$ is often a singular scheme, and the intersection theory of such spaces is hard to deal with.
	There is one case in which all the relevant computations can be pushed onto a smooth variety, which simplifies the task considerably. Assume that $\mathcal{A}$ is a smooth variety with an action of $\mathbb{C}^\ast$ such that
	\begin{align}\label{ProperFixedLocusHp}
		\text{the fixed subvariety } \mathcal{A}^{\mathbb{C}^\ast} \text{ is proper}.
	\end{align}
	This ${\mathbb{C}^\ast}$-action defines a grading on the ring of regular functions of $\mathcal{A}$. Given a function $\varphi \in H^0(\mathcal{A}, \mathcal{O}_\mathcal{A})$ of degree $d$, the critical locus $X:= \text{Crit}(\varphi)$ is canonically endowed with an equivariant perfect obstruction theory of virtual dimension zero, whose virtual tangent bundle is represented by the complex concentrated in [0,1]
	\begin{align}\label{virtualTangentBundle}
		\left[T_\mathcal{A} \xrightarrow{\nabla d \varphi} \Omega_\mathcal{A} \otimes \mathfrak{s}^d \right]_{\vert X}.
	\end{align}
	Here $\mathfrak{s}$ denotes the 1-dimensional representation of $\mathbb{C}^\ast$ of weight 1 and $\nabla d\varphi$ is the vertical derivative of the section $d\varphi$. We aim to show that the computation of the invariants of $X$ can be pushed forward to a computation on the smooth proper variety $\mathcal{A}^{\mathbb{C}^\ast}$. First of all, we know how to push forward the virtual fundamental class: in this "global" case the virtual fundamental cycle of $X$ is the refined Euler class
	\begin{align*}
		[X]^{\text{vir}} = e^{\mathbb{C}^\ast}_{\text{ref}}(\Omega_\mathcal{A}\otimes \mathfrak{s}^d)
	\end{align*}
	and, as proven in \cite{Graber}, the localized virtual fundamental class pushes forward, via the inclusion  $X^{\mathbb{C}^\ast} \hookrightarrow \mathcal{A}^{\mathbb{C}^\ast}$, to
	\begin{align*}
		i_\ast \left(\frac{[X^{\mathbb{C}^\ast}]^{\text{vir}}}{e^{\mathbb{C}^\ast}(\mathcal{N}^{\text{vir}})}\right) = \frac{e^{\mathbb{C}^\ast}(\Omega_\mathcal{A}\otimes \mathfrak{s}^d)}{e^{\mathbb{C}^\ast}(\mathcal{N}_{\mathcal{A}^{\mathbb{C}^\ast}/\mathcal{A}})}.
	\end{align*}
	In particular, this implies the following
	\begin{lem}
		Let $E \in K_{\mathbb{C}^\ast}^0(\mathcal{A})$. Then
		\begin{align*}
			\chi^{\text{vir}}(X, E_{\vert X}) = \chi\left(\mathcal{A}, E \otimes \Lambda_{-y^{-d}}T_\mathcal{A}\right)
		\end{align*}
	\end{lem}
	\begin{proof}
		By the Hirzebruch-Riemann-Roch theorem we can write the Euler characteristic as an integral over the virtual fundamental class and we use push-pull to push the integral onto $\mathcal{A}^{\mathbb{C}^\ast}$. In formulae:
		\begin{align*}
			\chi^{\text{vir}}(X, E_{\vert X}) &= \int_{[X^{\mathbb{C}^\ast}]^{\text{vir}}} \frac{\text{ch}^{\mathbb{C}^\ast}(i^\ast E) \text{td}^{\mathbb{C}^\ast}(T^{\text{vir}}_X)}{e^{\mathbb{C}^\ast}(\mathcal{N}^{\text{vir}})}\\
			&= \int_{A^{\mathbb{C}^\ast}} \text{ch}^{\mathbb{C}^\ast}(i^\ast E) \frac{e^{\mathbb{C}^\ast}(\Omega_\mathcal{A}\otimes \mathfrak{s}^d)}{\text{td}^{\mathbb{C}^\ast}(\Omega_\mathcal{A}\otimes \mathfrak{s}^d)}
			\frac{\text{td}^{\mathbb{C}^\ast}(T_\mathcal{A})}{e^{\mathbb{C}^\ast}(\mathcal{N}_{\mathcal{A}^{\mathbb{C}^\ast}/\mathcal{A}})}\\
			&=  \int_{A^{\mathbb{C}^\ast}} \text{ch}^{\mathbb{C}^\ast}(i^\ast E) \text{ch}^{\mathbb{C}^\ast}(\Lambda_{-1}(T_\mathcal{A} \otimes \mathfrak{s}^{-d}))	\frac{\text{td}^{\mathbb{C}^\ast}(T_\mathcal{A})}{e^{\mathbb{C}^\ast}(\mathcal{N}_{\mathcal{A}^{\mathbb{C}^\ast}/\mathcal{A}})}\\
			&= \chi\left(\mathcal{A}, E \otimes \Lambda_{-1}(T_\mathcal{A}\otimes \mathfrak{s}^{-d})\right)
		\end{align*}
	\end{proof}
	Let's use these results to push the computations onto $\mathcal{A}^{\mathbb{C}^\ast}$:
	\begin{lem}\label{invariantsOnA}
		Denote with $T_\mathcal{A}^{\text{vir}}$ the class $T_\mathcal{A} - y^d \Omega_\mathcal{A}$ in $K_{\mathbb{C}^\ast}^0(\mathcal{A})$.
		The invariants of $X$ can be computed as
		\begin{align*}
			&\text{DT} = \int_{\mathcal{A}}{e^{\mathbb{C}^\ast}(\Omega_\mathcal{A}\otimes \mathfrak{s}^d)},\\
			&\chi_y = (-y^{-\frac{d}{2}})^{\text{dim}\mathcal{A}} \chi\left(\mathcal{A}, \Lambda_{-y^d} \Omega_\mathcal{A} \right),\\
			&\text{Ell}(q, y) = \chi\left(\mathcal{A}, \mathcal{E}_{1/2}(T_\mathcal{A}^{\text{vir}}) \otimes \sqrt{\text{det}(T_\mathcal{A}^{\text{vir}})^\vee} \otimes \Lambda_{-y^{-d}} T_\mathcal{A} \right).
		\end{align*}
	\end{lem}
	\begin{proof}
		By definition we have that $T^\text{vir}_X = {T_\mathcal{A}^{\text{vir}}}_{\vert X}$ and
		\begin{align*}
			\mathcal{E}_{1/2}(T^{\text{vir}}_X) = \mathcal{E}_{1/2}(T_\mathcal{A}^{\text{vir}})_{\vert X} \quad \text{ and } \quad K^{\text{vir}}_X = \text{det}(\text{T}_\mathcal{A}^\text{vir})^\vee_{\vert X} = K_\mathcal{A}^{\otimes 2} \otimes \mathfrak{s}^{d \, \text{dim}\mathcal{A}}_{\vert X}.
		\end{align*}
		Using the previous lemma and remembering (for the Hirzebruch genus) that in the K-theory of $\mathcal{A}$
		\begin{align*}
			K_\mathcal{A} \otimes \Lambda_{-y^{-d}} T_\mathcal{A} = (-y^{-d})^{\text{dim}\mathcal{A}} \Lambda_{-y^{d}} \Omega_\mathcal{A},
		\end{align*}
		we can rewrite our invariants as claimed.
	\end{proof}
	\subsection{Invariants of critical loci of quotients.}\label{subsPullbackOnAbelian}
	In case $\mathcal{A}$ is the quotient of a smooth variety by the action of a reductive algebraic group we can lift the computations of our invariants onto the intermediate quotient by the maximal torus. Assume that there is a smooth variety $V$ admitting the action of a reductive algebraic group $G$ and a regular linearization $\mathcal{L}$, so that
	\begin{align*}
		V^{G\text{-sst}} = V^{G\text{-st}} \text{ and } V^{T\text{-sst}} = V^{T\text{-st}}
	\end{align*}
	where $T\subseteq G$ is a maximal subtorus and assume that $\mathcal{A} \simeq V/\!/G$. Let $V$ admit a ${\mathbb{C}^\ast}$-action commuting with the $G$-action and inducing the original ${\mathbb{C}^\ast}$-action on $\mathcal{A}$. Let $V^\prime \subseteq V$ be an invariant subvariety so that $\mathcal{A}^{\mathbb{C}^\ast} \simeq V^\prime /\!/ G$ and consider the intermediate quotients 
	\begin{align*}
		\mathcal{B}:= V/\!/T, \qquad \mathcal{B}^\prime := V^\prime /\!/T.
	\end{align*}
	We can consider the projection
	\begin{align*}
		\pi : \mathcal{B}^{G\text{-sst}} \rightarrow \mathcal{A}
	\end{align*}
	and the inclusion $j : B^{G\text{-sst}} \hookrightarrow B$. The equivariant Martin's formula of corollary \ref{equivariantMartinFormula} has the following direct application:
	\begin{lem}
		Let $\alpha \in A_{\mathbb{C}^\ast}^\ast(\mathcal{A}^{\mathbb{C}^\ast})$ and $\beta \in A_{\mathbb{C}^\ast}^\ast(\mathcal{B}^\prime)$ be such that $\pi^\ast \alpha = j^\ast \beta$. Then
		\begin{align*}
			\int_{\mathcal{A}^{\mathbb{C}^\ast}}\alpha = \frac{1}{\vert W \vert} \int_{\mathcal{B}^\prime} \beta \cdot e^{\mathbb{C}^\ast}(\mathcal{R}).
		\end{align*}
	\end{lem}
	In particular, we can compute the integral on $\mathcal{B}^\prime$ by Atiyah-Bott localization. We first study the fixed locus on $\mathcal{B}^\prime$:
	\begin{lem}
		Let $F$ be a connected component of $\mathcal{B}^{\mathbb{C}^\ast}$ containing at least a $G$-semistable point. Then $F$ is a connected component of $(\mathcal{B}^\prime)^{\mathbb{C}^\ast}$.
	\end{lem}
	\begin{proof}
		Notice that we just have to prove that $F$ is contained in $\mathcal{B}^\prime$, but this follows from the fact that $\pi(F)$ is nonempty and fixed in $\mathcal{A}$.
	\end{proof}
	This, together with corollary \ref{zeroIntegral}, gives us the following result:
	\begin{pro}
		Let $\alpha \in A_{\mathbb{C}^\ast}^\ast(\mathcal{A})$ and $\beta \in A_{\mathbb{C}^\ast}^\ast(\mathcal{B})$ be such that $\pi^\ast \alpha = j^\ast \beta$. Then
		\begin{align*}
			\int_{\mathcal{A}^{\mathbb{C}^\ast}} \frac{\alpha}{e^{\mathbb{C}^\ast}(\mathcal{N}_{\mathcal{A}^{\mathbb{C}^\ast}/\mathcal{A}})} = \frac{1}{\vert W \vert} \int_{\mathcal{B}^{\mathbb{C}^\ast}} \frac{\beta}{e^{\mathbb{C}^\ast}(\mathcal{N}_{\mathcal{B}^{\mathbb{C}^\ast}/\mathcal{B}})} e^{\mathbb{C}^\ast}(\mathcal{R}).
		\end{align*}
	\end{pro}
	The $K$-theoretic analogue of this is the following
	\begin{cor}
		Let $E \in K_{\mathbb{C}^\ast}^0(\mathcal{A})$ and $F \in K_{\mathbb{C}^\ast}^0(\mathcal{B})$ be such that $\pi^\ast E = j^\ast F$. Then
		\begin{align*}
			\chi(\mathcal{A}, E) = \chi(\mathcal{B}, F \otimes \Lambda_{-1} \mathcal{R}).
		\end{align*}
	\end{cor}
	\begin{proof}
		This is a simple combination of Hirzebruch-Riemann-Roch and the previous result.
	\end{proof}
	We use this result to pull the computations for our invariants back from $\mathcal{A}^{\mathbb{C}^\ast}$ to the intermediate quotient $\mathcal{B}^{\mathbb{C}^\ast}$:
	\begin{lem}\label{invariantsOnB}
		Set $T_\mathcal{B}^\text{vir} \in K_{\mathbb{C}^\ast}^0(\mathcal{B})$ as
		\begin{align*}
			T^{\text{vir}}_\mathcal{B} := T_\mathcal{B} - y^d \Omega_\mathcal{B} - \mathcal{R} + y^d \mathcal{R}.
		\end{align*}
		As long as $\mathcal{R} \otimes \mathfrak{s}^d$ has trivial fixed part on every connected component of $\mathcal{B}^{\mathbb{C}^\ast}$, the invariants of $X$ can be computed as
		\begin{align*}
			&\text{DT} = \vert W \vert^{-1} \int_{\mathcal{B}}{e^{\mathbb{C}^\ast}(T_\mathcal{B}-T^{\text{vir}}_\mathcal{B})}\\
			&\chi_y = (-y^{-\frac{d}{2}})^{\text{dim}\mathcal{A}}  \vert W \vert^{-1}\chi\left(\mathcal{B}, \Lambda_{-1} (T_\mathcal{B}-T^{\text{vir}}_\mathcal{B}) \right),\\
			&\text{Ell}(q, y) = \vert W \vert^{-1}\chi\left(\mathcal{B}, \mathcal{E}_{1/2}(T^{\text{vir}}_\mathcal{B}) \otimes \sqrt{\text{det}(T^{\text{vir}}_\mathcal{B})^\vee} \otimes  \Lambda_{-1} (\Omega_\mathcal{B}-(T^{\text{vir}}_\mathcal{B})^\vee)\right).
		\end{align*}
	\end{lem}
	\begin{proof}
		Over $\mathcal{B}^{G\text{-sst}}$ we have the short exact sequence of ${\mathbb{C}^\ast}$-equivariant bundles
		\begin{align*}
			0 \rightarrow \mathcal{R} \rightarrow T_\mathcal{B} \rightarrow \pi^\ast T_\mathcal{A} \rightarrow 0,
		\end{align*}
		hence in $K_{\mathbb{C}^\ast}^0(\mathcal{B})$ we have that
		\begin{align*}
			\pi^\ast T_\mathcal{A} = T_\mathcal{B} - \mathcal{R} \quad \text{and} \quad \pi^\ast \mathcal{N}_{\mathcal{A}^{\mathbb{C}^\ast}/\mathcal{A}} = \mathcal{N}_{\mathcal{B}^{\mathbb{C}^\ast}/\mathcal{B}}.
		\end{align*}
		This shows that $\pi^\ast T^{\text{vir}}_\mathcal{B} = j^\ast T_\mathcal{A}^{\text{vir}}$ and hence we can rewrite our invariants as claimed.
	\end{proof}	
	\subsection{Statement of the result.}\label{statement}
	Here we state our main result, which holds in the following context. Let $G$ be a reductive algebraic group, with maximal subtorus $T$, set of roots $\mathfrak{R}\subset \mathfrak{t}_\mathbb{Z}^\vee$ and Weyl group $W$, acting on a complex linear space $V$. Let $\mathcal{W} \subset \mathfrak{t}_\mathbb{Z}^\vee$ be the set of weights for this action and assume that $\xi \in (\mathfrak{t}_\mathbb{Z}^\vee)^W$ is a regular stability with a sum-regular perturbation $\tilde{\xi}$. Consider the decomposition of $V$ into $T$-invariant spaces $V_i \simeq \mathbb{C}$
	\begin{align*}
		V \simeq \bigoplus_{i=1}^{\text{dim}V} V_i,
	\end{align*}
	so that $T$ acts through a morphism to $(\mathbb{C}^\ast)^{\text{dim}V}$. For every index $i$, let $\rho_i \in \mathcal{W}$ be the weight of $T \curvearrowright V_i$. Let $\mathbb{C}^\ast$ act on $V$ with $R$-charge (set of weights) $R \in \mathbb{Z}^{\text{dim}V}$ and assume the action commutes with the one of $G$. Set $\mathcal{A} := V/\!/G$ and let $X$ be the critical locus of a degree $d$ function on $\mathcal{A}$. Assume moreover that 
	\begin{align}\label{hpProperness}
		\text{the fixed locus }\mathcal{A}^{\mathbb{C}^\ast}\text{ is proper.}
	\end{align}
	For every index $i \in \lbrace 1, ..., \text{dim}V\rbrace$ consider the affine function $f_i$ on the complex Lie algebra $\mathfrak{t}$ given by
	\begin{align*}
		f_i(u):= \rho_i(u) + R_i
	\end{align*}
	and let $\mathfrak{M}^\xi\subset \mathfrak{t}$ be the set of stable isolated intersections of the corresponding hyperplanes $V(f_i)$ (as in definitions \ref{isolatedInt} and \ref{stableIsolatedInt}). Assume that
	\begin{align}\label{InvertibilityEulerClasses}
		\text{For every intersection } P\in \mathfrak{M}^\xi \text{ and every root } \alpha \in \mathfrak{R}, \text{ } d+\alpha(P) \neq 0.
	\end{align}
	For every $P \in \mathfrak{M}^\xi$ let $I_P \subset \lbrace1, ..., \text{dim}V\rbrace$ be the set of indices $i$ so that $f_i(P)=0$. Let $\mathcal{W}_P \subseteq \mathcal{W}$ be the set of weights corresponding to $I_P$. 
	The main result of the paper is expressed in terms of the following meromorphic functions defined on the complex Lie algebra $\mathfrak{t}$:
	\begin{align}\label{functions}
		\begin{split}
			&Z_s= \left(\frac{1}{ds}\right)^{\text{dim}T} \prod_{\alpha \in \mathfrak{R}} \frac{\alpha}{(ds+\alpha)} \prod_{i=1}^{\text{dim}V} \frac{(d-R_i) s - \rho_i}{R_i s + \rho_i},\\
			&Z_\hbar = \left(\frac{\pi \hbar}{\sin(d\pi \hbar)}\right)^{\text{dim}T} \prod_{\alpha \in \mathfrak{R}} \frac{\sin(\pi \hbar \alpha)}{\sin(\pi \hbar(\alpha + d))} \prod_{i=1}^{\text{dim}V}\frac{\sin(\pi \hbar (d-R_i-\rho_i))}{\sin(\pi \hbar (R_i+\rho_i))},\\
			&Z_{\tau, \hbar} = \left(\frac{2\pi \hbar \, \eta(\tau)^3}{\theta(\tau \vert d\hbar)}\right)^{\text{dim}T} \prod_{\alpha \in \mathfrak{R}} \frac{\theta(\tau\vert \hbar \alpha)}{\theta(\tau\vert \hbar(d + \alpha))} \prod_{i=1}^{\text{dim}V} \frac{\theta(\tau\vert \hbar(d- R_i - \rho_i))}{\theta(\tau\vert \hbar(R_i + \rho_i))}.
		\end{split}
	\end{align}
	\begin{theorem}\label{mainTheorem}
		The virtual invariants of $X$ can be computed with the following formulae:
		\begin{align*}
			\text{DT} &= \frac{1}{\vert W \vert} \sum_{P \in \mathfrak{M}^\xi} \text{JK}_P^{\mathcal{W}_P}\left(Z_s, \tilde{\xi}\right), \qquad \forall s \in \mathbb{C}^\ast\\
			\chi_{e^{2 \pi i \hbar}} &= \frac{1}{\vert W \vert} \sum_{P \in \mathfrak{M}^\xi} \text{JK}_P^{\mathcal{W}_P}\left(Z_\hbar, \tilde{\xi}\right),\\
			\text{Ell}(e^{2 \pi i \tau}, e^{2\pi i \hbar}) &= \frac{1}{\vert W \vert} \sum_{P \in \mathfrak{M}^\xi} \text{JK}_P^{\mathcal{W}_P}\left(Z_{\tau, \hbar}, \tilde{\xi}\right).
		\end{align*}
	\end{theorem}
	\begin{rem}\label{projectivityOfInt}
		Condition (\ref{hpProperness}) can sometimes be known from geometric considerations on the GIT data. For example, we know that the quotient $V/\!/G$ fibers over the affine quotient through a projective morphism
		\begin{align*}
			p : V/\!/G \rightarrow \mathbb{A}^N,
		\end{align*}
		where this map is given as $(f_1, ..., f_N)$ and the regular functions $f_i$ are a system of generators of the ring of $G$-invariant functions on $V$. Now assume that the ${\mathbb{C}^\ast}$-action is such that all the $f_i$ are ${\mathbb{C}^\ast}$-equivariant with nonzero weight. Then the map $p$ above becomes ${\mathbb{C}^\ast}$-equivariant once we put the action on $\mathbb{A}^N$ scaling the components with the appropriate (nonzero!) weights. Since fixed points map to fixed points, the fixed locus $(V/\!/G)^{\mathbb{C}^\ast}$ maps to the origin, hence it is a closed subscheme of a fiber of $p$, which is projective.
	\end{rem}
	\begin{rem}
		Lemma \ref{ProjectiveTLoci} shows that if $G=T$ then hypothesis (\ref{hpProperness}) can be rephrased as
		\begin{align}\label{hpPropernessPrime}
			\text{The elements of } \mathfrak{M}^\xi \text{ are all projective.}
		\end{align}
	\end{rem}
	\begin{rem}
		As we will show in the following section, condition (\ref{InvertibilityEulerClasses}) is equivalent to the roots bundle $\mathcal{R}\otimes \mathfrak{s}^d$ having trivial fixed part on every connected component of $\mathcal{B}^{\mathbb{C}^\ast}$. This is needed in order to invert the relevant equivariant classes in lemma \ref{invariantsOnB}.
	\end{rem}
	\subsection{Proof of the formula.}\label{proof}
	Set $\mathcal{B}:= V/\!/T$, let $V_P \subseteq V$ be the sum of the weight spaces corresponding to $I_P$ and set $\mathcal{B}_P := V_P/\!/T$ as in section \ref{subsCombinatorics}.
	\begin{rem}
		We work with the additional hypothesis that all the intersections $P \in \mathfrak{M}^\xi$ are integral, namely $\rho(P) \in \mathbb{Z}$ for all $\rho \in \mathcal{W}$. For more details on how to reduce to this case see lemma \ref{reductionIntegralCase} and the discussion above it in the appendix.
	\end{rem}
	In lemma \ref{invariantsOnB} we have shown that the virtual invariants of $X$ can be computed as a sum, over the connected components of $\mathcal{B}^{\mathbb{C}^\ast}$, of some contributions.
	An isolated intersection $P \in \mathfrak{M}^\xi$ identifies a connected component $\mathcal{B}_P$ of $\mathcal{B}^{\mathbb{C}^\ast}$ by proposition \ref{fixedTlocus}. Here we compute explicitly the corresponding contribution to $\text{DT}$, ${\chi}_y$, and $\text{Ell}(q,y)$. 
	Thanks to corollary \ref{superDirectJK}, we know how to compute these contributions once we find a group homomorphism $\phi : {\mathbb{C}^\ast} \rightarrow T$ so that $s \cdot v = \phi(s)\cdot v$ for every $s \in {\mathbb{C}^\ast}$ and $v \in V_P$.
	\begin{lem}\label{sameAction}
		The morphism $\phi: {\mathbb{C}^\ast} \rightarrow (\mathbb{C}^\ast)^{\text{dim}V}$ defined by
		\begin{align*}
			\phi(s)_i := s^{-\rho_i(P)} \qquad \forall i \in \lbrace 1, ..., \text{dim}V\rbrace
		\end{align*}
		factors through the map $T \hookrightarrow (\mathbb{C}^\ast)^{\text{dim}V}$ and satisfies $\phi(s)_i = s^{R_i}$ for every $i \in I_P$, thus inducing the R-charge action (\ref{CircleAction}) of ${\mathbb{C}^\ast}$ on $V_P$.
	\end{lem}
	\begin{proof}
		Consider an element $s \in {\mathbb{C}^\ast} $ and define an element $t \in (\mathbb{C}^\ast)^{\text{dim}V}$ by
		\begin{align*}
			t_i := s^{-\rho_i(P)} \qquad \qquad \forall i \in \lbrace 1, ..., \text{dim}V\rbrace.
		\end{align*}
		Since $P \in \mathfrak{t}$ we have that $t \in T$. Since for every $i \in I_P$ we have $f_i(P) = 0$, the equality $R_i = -\rho_i(P)$ holds true.
	\end{proof}
	\begin{rem}
		Notice that this fact combined with lemma \ref{descentDiagram} shows that condition (\ref{InvertibilityEulerClasses}), namely that $\alpha(P) + d \neq 0$ for all roots $\alpha \in \mathfrak{R}$, is equivalent to the roots bundle $\mathcal{R} \otimes \mathfrak{s}^d$ having trivial fixed subbundle on $\mathcal{B}_P$.
	\end{rem}
	Consider the equivariant descent map 
	\begin{align*}
		\tilde{d} : K_{T \times {\mathbb{C}^\ast}}^0(V_P) \rightarrow K_{\mathbb{C}^\ast}^0(\mathcal{B}_P).
	\end{align*}
	Recall that this is inverse to the equivariant pullback $p^\ast$ along the quotient map $p : (V_P)^{T\text{-sst}}\rightarrow \mathcal{B}_P$. Since the $T$-action is locally free on the semistable locus we have that $p^\ast \mathcal{N}_{\mathcal{B}_P/\mathcal{B}} \simeq \mathcal{N}_{V_P/V}$, by definition we have that $p^\ast \mathcal{R} \simeq \mathfrak{g}/\mathfrak{t}$ and we have the short exact sequence on $V_P$ 
	\begin{align*}
		0 \rightarrow \mathfrak{t} \rightarrow T_V \rightarrow p^\ast T_\mathcal{B} \rightarrow 0.
	\end{align*}
	This implies that, if in $K_{T \times \mathbb{C}^\ast}^0(V)$ we set
	\begin{align*}
		E := y^d \Omega_V  - y^d \mathfrak{t} + \mathfrak{g}/\mathfrak{t} - y^d \mathfrak{g}/\mathfrak{t},
	\end{align*}
	then we can write
	\begin{align*}
		T^{\text{vir}}_\mathcal{B} = \tilde{d}\left(T_V - \mathfrak{t} - E \right).
	\end{align*}
	We can rewrite lemma \ref{invariantsOnB} as
	\begin{lem}\label{invariantsOnBandD}
		The contribution of the connected component $\mathcal{B}_P$ of $\mathcal{B}^{\mathbb{C}^\ast}$ to the DT invariant of $X$ can be computed as
		\begin{align*}
			\vert W \vert^{-1} \int_{\mathcal{B}_P} e^{{\mathbb{C}^\ast}}\left(\tilde{d}\left(E-\mathcal{N}_{V_P/V}\right)\right).
		\end{align*}
		For the virtual Hirzebruch genus $\chi_y$ it is
		\begin{align*}
			(-y^{-\frac{d}{2}})^{\text{dim}\mathcal{A}}  \vert W \vert^{-1}\chi\left(\mathcal{B}_P, \tilde{d}\left(\Lambda_{-1} \left(E-\mathcal{N}^\vee_{V_P/V}\right)\right) \right)
		\end{align*}
		while the contribution to the chiral elliptic genus $\text{Ell}(q, y)$ coincides with
		\begin{align*}
			&\vert W \vert^{-1}\chi\left(\mathcal{B}_P, \tilde{d}\left( \mathcal{E}_{1/2}(T_V-\mathfrak{t}-E) \otimes \sqrt{\text{det}(\Omega_V - E^\vee)} \otimes  \Lambda_{-1} \left(E^\vee - \mathcal{N}^\vee_{V_P/V}\right)\right)\right).
		\end{align*}
	\end{lem}
	Now we are ready to prove our main result.
	\begin{proof}[Proof of theorem \ref{mainTheorem}]
		Here we combine lemma \ref{sameAction} and corollary \ref{superDirectJK} in order to compute the invariants from their expression given in lemma \ref{invariantsOnBandD}.\\
		$\bullet$ By corollary \ref{superDirectJK}, the contribution of $P \in \mathfrak{M}^\xi$ to the DT invariant becomes
			\begin{align*}
				\frac{1}{ \vert W \vert}\text{JK}_{Ps}^{\mathcal{W}_P}\left(\frac{1}{e^{T \times {\mathbb{C}^\ast}}(\mathfrak{t}\otimes \mathfrak{s}^d)}
				\frac{e^{T\times {\mathbb{C}^\ast}}(\mathfrak{g}/\mathfrak{t})}{e^{T\times {\mathbb{C}^\ast}}(\mathfrak{g}/\mathfrak{t} \otimes \mathfrak{s}^d)}
				\frac{e^{T \times {\mathbb{C}^\ast}}(\Omega_V \otimes \mathfrak{s}^d)}{e^{T\times {\mathbb{C}^\ast}}(T_{V})}, \tilde{\xi} \right).
			\end{align*}
			The torus $T \times {\mathbb{C}^\ast}$ acts trivially on $\mathfrak{t}$, hence $e^{T\times {\mathbb{C}^\ast}}(\mathfrak{t}\otimes \mathfrak{s}^d) = (ds)^{\text{dim}T}$. The equivariant Chern roots of $T_V$ are of the form $\rho_i + R_i s$ for $i \in \lbrace 1, ..., \text{dim}V\rbrace$. Those of $\mathfrak{g}/\mathfrak{t}$ are given by the roots $\alpha \in \mathfrak{R}$, which completes the first part of the proof.\\$\,$\\
			$\bullet$ By the same corollary \ref{superDirectJK}, the contribution of $P$ to $\vert W \vert \chi_{e^s}$ is equal to $\text{JK}_{Ps}^{\mathcal{W}_P}$ of
			\begin{align*}
				(-e^{-\frac{ds}{2}})^{\text{dim}\mathcal{A}} \text{ch}^{T \times {\mathbb{C}^\ast}}\left(\frac{1}{\Lambda_{-y^d}\mathfrak{t}} \otimes \frac{\Lambda_{-1}\mathfrak{g}/\mathfrak{t}}{\Lambda_{-y^d}\mathfrak{g}/\mathfrak{t} } \otimes
				\frac{\Lambda_{-y^d}\Omega_V }{\Lambda_{-1}\Omega_{V}} \right).
			\end{align*}
			Using $\text{ch}^{T \times {\mathbb{C}^\ast}}(\Lambda_{-y^d} \mathfrak{t}) = (1-e^{ds})^{\text{dim}T}$ and lemma \ref{computationsLambda}, this contribution is equal to the residue $\text{JK}_{Ps}^{\mathcal{W}_P}$ of
			\begin{align*}
				&(-1)^{\text{dim}\mathcal{A}+\text{dim}T} \left(\frac{1}{2 \sinh(\frac{ds}{2})}\right)^{\text{dim}T} \prod_{\alpha \in \mathfrak{R}} \frac{\sinh(\frac{\alpha}{2})}{\sinh(\frac{\alpha + ds}{2})} \prod_{i = 1}^{\text{dim(V)}} \frac{\sinh(\frac{\rho_i + (R_i-d)s}{2})}{\sinh(\frac{\rho_i + R_i s}{2})}.
			\end{align*}
			where we have used that $\text{dim}\mathcal{A} = \text{dim}V - \text{dim}G$.
			Set now $\hbar = \frac{s}{2\pi i}$. If we rescale the coordinates on $\mathfrak{t}$ by ${2\pi i \hbar}$, then by lemma \ref{rescalingJK} a term $(2\pi i \hbar)^{\text{dim}T}$ appears in front of the residue. The contribution of $P$ to the genus is $\text{JK}_{P}^{\mathcal{W}_P}$ of
			\begin{align*}
				&(-1)^{\text{dim}\mathcal{A}+\text{dim}T} \left(\frac{\pi \hbar}{\sin(d\pi \hbar)}\right)^{\text{dim}T}\prod_{\alpha \in \mathfrak{R}} \frac{\sin(\pi \hbar \alpha)}{\sin(\pi \hbar(\alpha + d))} \prod_{i=1}^{\text{dim}V} \frac{\sin(\pi \hbar (\rho_i + R_i-d))}{\sin(\pi \hbar (\rho_i + R_i))}.
			\end{align*}
			Finally, if we swap the sign of the terms $\rho_i + R_i-d$ we add a term $(-1)^{\text{dim}V}$ in front, and since
			\begin{align*}
				\text{dim}\mathcal{A} + \text{dim}V + \text{dim}T \equiv_2 \text{dim}G - \text{dim}T \equiv_2 0,
			\end{align*}
			the overall minus sign disappears, completing the second step of the proof.\\$\,$\\
			$\bullet$ Again by corollary \ref{superDirectJK}, the contribution of $P \in \mathfrak{M}^\xi$ to $\vert W \vert \text{Ell}(q, e^s)$ is equal to the residue $\text{JK}_{Ps}^{\mathcal{W}_P}$ of 
			\begin{align*}
				\text{ch}^{T\times {\mathbb{C}^\ast}} \left(\frac{\hat{\mathcal{E}}_{1/2}(y^d\mathfrak{t})}{\mathcal{E}_{1/2}(\mathfrak{t})}\otimes \frac{\hat{\mathcal{E}}_{1/2}(y^d \mathfrak{g}/\mathfrak{t})}{\hat{\mathcal{E}}_{1/2}(\mathfrak{g}/\mathfrak{t})}\otimes \frac{\hat{\mathcal{E}}_{1/2}(T_V)}{\hat{\mathcal{E}}_{1/2}(y^d\Omega_V)}\right).
			\end{align*}
			By lemma \ref{computationsE}, we find that this quantity coincides with
			\begin{align*}
				\left(\frac{-i \eta(q)^3}{\theta(q ; e^{ds})}\right)^{\text{dim}T} \prod_{\alpha \in \mathfrak{R}} \frac{\theta(q ; e^{\alpha})}{\theta(q ; e^{ds + \alpha})} 
				\prod_{i=1}^{\text{dim}V} \frac{\theta(q ; e^{(d- R_i)s - \rho_i})}{\theta(q ; e^{R_i s + \rho_i})}.
			\end{align*}
			Set now $\hbar = \frac{s}{2\pi i}$. If we rescale the coordinates on $\mathfrak{t}$ by ${2\pi i \hbar}$, then by lemma \ref{rescalingJK} a term $(2\pi i \hbar)^{\text{dim}T}$ appears in front of the residue. The genus $\vert W \vert \text{Ell}(q, e^s)$ is equal to the sum over $P \in \mathfrak{M}^\xi$ of the residues $\text{JK}_{P}^{\mathcal{W}_P}$ of
			\begin{align*}
				\left(\frac{2\pi \hbar \, \eta(q)^3}{\theta(q ; e^{2 d \pi i \hbar})}\right)^{\text{dim}T} \prod_{\alpha \in \mathfrak{R}} \frac{\theta(q ; e^{2\pi i \hbar \alpha})}{\theta(q ; e^{2\pi i \hbar(d + \alpha)})} \prod_{i = 1}^{\text{dim}V} \frac{\theta(q ; e^{2\pi i \hbar(d- R_i - \rho_i)})}{\theta(q ; e^{2\pi i \hbar(R_i + \rho_i)})}.
			\end{align*}
			Setting $q = e^{2 \pi i \tau}$ completes the proof.
	\end{proof}
	\begin{rem}
		As expected by proposition \ref{convergence}, once we set $Z := Z_{s=1}$ the power series we have just defined satisfy
		\begin{align*}
			Z_{\hbar} = \lim_{\tau \rightarrow i \infty} Z_{\tau, \hbar} \qquad \text{and} \qquad Z = \lim_{\hbar \rightarrow 0} Z_\hbar.
		\end{align*}
		This fact could be used to  give a different proof of proposition \ref{convergence} once one recalls that in this case the JK residue commutes with the limit as shown in lemma 3.10 of \cite{OntaniStoppa}.
	\end{rem}
	\section{Applications and examples.}\label{sectionExamples}
	Here we discuss some consequences of theorem \ref{mainTheorem}.
	\subsection{Classical invariants of smooth projective varieties.}\label{SubsectionProjVar}
	The result we have proven gives a formula for computing classical invariants of many projective varieties. 
	Recall that if $X$ is a smooth projective variety there are, among others, three nice invariants: the \textit{Euler number}, the \textit{Hirzebruch genus} and the \textit{elliptic genus}:
	\begin{align*}
		&\chi(X) = \int_X e(T_X),\\
		&\chi_y(X) = \chi(X, \Lambda_y\Omega_X),\\
		&\text{Ell}_{q,y}(X)= \chi(X, \mathcal{E}(T_X) \otimes \Lambda_{-y}\Omega_X),
	\end{align*}
	where the class $\mathcal{E}(V)\in K^0(X)[\![q, y]\!]_y$ is defined, for every $V \in K^0(X)$, as
	\begin{align*}
		\mathcal{E}(V) :&= \bigotimes_{n\geq 1} \left(\text{Sym}_{q^n}(V \oplus V^\vee)\otimes\Lambda_{-y^{-1}q^n}(V)\otimes\Lambda_{-y q^n}(V^\vee)\right)\\
		&= \bigotimes_{n\geq 1}\left( \frac{\text{Sym}_{q^n}(V \oplus V^\vee)}{\text{Sym}_{y^{-1}q^n}(V)\otimes\text{Sym}_{y q^n}(V^\vee)}\right) = \mathcal{E}_{1/2}((1-y^{-1})V).
	\end{align*}
	\begin{rem}
		Here we consider the elliptic genus as defined, for example, in equation (14), page 175 of \cite{HirzebruchModularForms}.
	\end{rem}
	Assume that $\mathcal{A}$ is a projective variety built as GIT quotient of a linear space $V$ by a $G$-action for a regular stability $\xi \in (\mathfrak{t}_\mathbb{Z}^\vee)^W$. Assume that $\mathbb{C}^\ast$ acts on $\mathcal{A}$. For every $d \in \mathbb{Z}\setminus \lbrace 0 \rbrace$ there is a unique regular function $\varphi : \mathcal{A} \rightarrow \mathbb{C}$ which is of degree $d$, namely $\varphi=0$. The critical locus is the entire $\mathcal{A}$ and the corresponding virtual tangent bundle is $[T_\mathcal{A} \rightarrow \Omega_\mathcal{A} \otimes \mathfrak{s}^d]$, whose $K$-theory class is $T^{\text{vir}}_\mathcal{A} = T_\mathcal{A} - y^d \Omega_\mathcal{A}$. 
	\begin{pro}
		Let $\mathcal{A}$ be endowed with the perfect obstruction theory given by $\mathcal{A}= \text{Crit}(0)$, where $0$ is thought as a function of degree $d$. Then
		\begin{align*}
			&\text{DT} = (-1)^{\text{dim}\mathcal{A}} \chi(\mathcal{A}),\\
			&\chi_y = (-y^{\frac{d}{2}})^{-\text{dim}\mathcal{A}} \chi_{-y^d}(\mathcal{A}),\\
			&\text{Ell}(q,y) = (-y^{\frac{d}{2}})^{-\text{dim}\mathcal{A}} \text{Ell}_{q,y^d}(\mathcal{A}).
		\end{align*}
		In particular, these invariants are independent of the choice of the ${\mathbb{C}^\ast}$-action.
	\end{pro}
	\begin{proof}
		By lemma \ref{invariantsOnA} and the localization formula the invariants are
		\begin{align*}
			&\text{DT} = \int_{\mathcal{A}} e^{\mathbb{C}^\ast}(\Omega_\mathcal{A}\otimes \mathfrak{s}^d),\\
			&\chi_y = (-y^{-\frac{d}{2}})^{\text{dim}\mathcal{A}} \chi\left(\mathcal{A}, \Lambda_{-y^d} \Omega_\mathcal{A} \right),\\
			&\text{Ell}(q, y) = (-y^{\frac{d}{2}})^{\text{dim}\mathcal{A}} \chi\left(\mathcal{A}, \mathcal{E}_{1/2}(T_\mathcal{A}-y^d \Omega_\mathcal{A}) \otimes K_\mathcal{A} \otimes \Lambda_{-y^{-d}} T_\mathcal{A}\right).
		\end{align*}
		and we can use 
		\begin{align*}
			K_\mathcal{A} \otimes \Lambda_{-y^{-d}} T_\mathcal{A} = (-y^{-d})^{\text{dim}\mathcal{A}} \Lambda_{-y^{d}} \Omega_\mathcal{A},
		\end{align*}
		to write the elliptic genus as
		\begin{align*}
			\text{Ell}(q, y) = (-y^{\frac{d}{2}})^{-\text{dim}\mathcal{A}} \chi(\mathcal{A}, \mathcal{E}(T_\mathcal{A}) \otimes \Lambda_{-y^d}\Omega_\mathcal{A}).
		\end{align*}
		The DT invariant is constant and hence it coincides with the integral of the Euler class of the cotangent bundle.
		On the other hand, the ${\mathbb{C}^\ast}$-equivariant Euler characteristic $\chi\left(\mathcal{A}, \Lambda_{-y^d} \Omega_\mathcal{A} \right)$ coincides with the non equivariant one since the ${\mathbb{C}^\ast}$-action on the sheaf cohomology groups $H^p\left(\mathcal{A},\, \Omega_\mathcal{A}^q \right)$ induced by ${\mathbb{C}^\ast} \curvearrowright \mathcal{A}$ is trivial. This is true since these are isomorphic to the Dolbeault cohomology groups where the action is given by pullback, but all the automorphisms induced by ${\mathbb{C}^\ast}$ are homotopic to the identity. The fact that the $\mathbb{C}^\ast$-equivariant elliptic genus of a projective variety coincides with the classical elliptic genus is the rigidity theorem conjectured by Witten \cite{WittenEllipticGenus}, proven for spin manifolds by Bott and Taubes \cite{BottTaubes} and in the general case by Hirzebruch (theorem at page 181 of \cite{HirzebruchModularForms}). 
	\end{proof}
	\begin{rem}
		Since we have proven that, in this case, the virtual invariants independent of the choice of the ${\mathbb{C}^\ast}$-action, we know that when we apply theorem \ref{mainTheorem} to compute them the result is independent on the choice of the $R$-charge (as long as it satisfies hypothesis (\ref{InvertibilityEulerClasses})).
	\end{rem}
	\subsection{Classical invariants of complete intersections.}\label{SubsectionCompleteInt}
	In some cases, theorem \ref{mainTheorem} becomes a useful tool to compute enumerative invariants of complete intersections in projective quotients of linear spaces. 
	Assume $Y := U/\!/G$ is a smooth projective variety built as a GIT quotient of a complex linear space $U$ by a reductive algebraic group $G$ for a regular stability $\xi$. Let $\mathcal{A} \rightarrow Y$ be a vector bundle on $Y$ induced from a $G$-equivariant bundle
	\begin{align*}
		V := U \oplus K \xrightarrow{\pi_U} U
	\end{align*}
	upstairs, where $K$ is a representation of $G$. Assume that, pulling back the linearization $\mathcal{L}$ from $U$ to $V:= U \oplus K$, the GIT quotient of $V$ is the total space of the vector bundle itself
	\begin{align}\label{HpVectorBundle}
		\mathcal{A} \simeq V/\!/G.
	\end{align}  
	Consider the $\mathbb{C}^\ast$-action on $V$ given by the trivial action on $U$ and by rescaling $K$ with weight $1$ (so $R_U=0$ and $R_K=1$ in terms of $R$-charges).
	\begin{lem}
		For every $d \geq 1$, the hypotheses of theorem \ref{mainTheorem} are satisfied in this context.
	\end{lem}
	\begin{proof}
		We are in the context described in remark \ref{CommutationExample}, hence the ${\mathbb{C}^\ast}$-action commutes with the $G$-action. The fixed locus $\mathcal{A}^{\mathbb{C}^\ast} = Y$ is proper by hypothesis. Notice that $\mathcal{B}^{\mathbb{C}^\ast}$ maps, via $\pi$, to $\mathcal{A}^{\mathbb{C}^\ast} = Y$, hence for every representative $v \in V$ of an element in $\mathcal{B}^{\mathbb{C}^\ast}$, the component $v_K$ in $K$ is zero. This means that every stable isolated intersection $P \in \mathfrak{M}^\xi$ is so that the only affine functions $f_i$ vanishing on $P$ come from the action on $U$. Since $R_U=0$ these functions satisfy $f_i = \rho_i$ and since the roots $\alpha \in \mathfrak{R}$ are linear combinations of the weights for the $U$-action, we have that
		\begin{align*}
			\alpha(P) + d = d \neq 0 \qquad \forall \alpha \in \mathfrak{R},
		\end{align*}
		thus satisfying condition (\ref{InvertibilityEulerClasses}).
	\end{proof}
	This means that we can invoke theorem \ref{mainTheorem} to compute the virtual invariants of the critical locus of a degree $d$ potential. In particular, we can focus on what happens in degree $d=1$. 
	\begin{pro}\label{completeIntInvariants}
		In the hypotheses above, the virtual invariants of the critical locus of a degree $d=1$ potential on $\mathcal{A}$ coincide with the classical invariants of the zero locus $X$ of a transversal section of the dual bundle $\mathcal{A}^\vee \rightarrow Y$:
		\begin{align*}
			&\text{DT} = (-1)^{\text{dim}X} \chi(X),\\
			&\chi_y = (-y^{\frac{1}{2}})^{-\text{dim}X} \chi_{-y}(X),\\
			&\text{Ell}(q,y) = (-y^{\frac{1}{2}})^{-\text{dim}X} \text{Ell}_{q,y}(X).
		\end{align*}
	\end{pro}
	\begin{proof}
		We use lemma \ref{invariantsOnA} and Poincaré duality (namely, if $i : X \hookrightarrow Y$ is the inclusion, the fact that $i_\ast i^\ast$ is multiplication by $e(\mathcal{A}^\vee)$), to express the virtual invariants as intersection numbers on $X$.	In the equivariant $K$-theory $K_{\mathbb{C}^\ast}^0(Y) \simeq K^0(Y)[y]_y$ we have that
		\begin{align*}
			\Omega_A \otimes \mathfrak{s} = y \Omega_Y + \mathcal{A}^\vee.
		\end{align*}
		This means that
		\begin{align*}
			\Omega_A \otimes \mathfrak{s} - \mathcal{N}_{Y/\mathcal{A}} = y \Omega_Y - y \mathcal{A} + \mathcal{A}^\vee.
		\end{align*}
		Notice that the conormal exact sequence
		\begin{align*}
			0 \rightarrow \mathcal{A} \rightarrow \Omega_X \rightarrow \Omega_Y \rightarrow 0
		\end{align*}
		implies, in $K_{\mathbb{C}^\ast}^0(X)$, that
		\begin{align*}
			y \Omega_Y - y \mathcal{A} = y \Omega_X.
		\end{align*}
		Hence
		\begin{align*}
			\text{DT} &= \int_Y e^{\mathbb{C}^\ast}\left(y \Omega_Y - y \mathcal{A} + \mathcal{A}^\vee\right) = \int_Y \frac{c_s(\Omega_Y)}{c_s(\mathcal{A})}e(\mathcal{A}^\vee) = \int_X \frac{c_s(\Omega_Y)_{\vert X}}{c_s(\mathcal{A})_{\vert X}}\\
			&= \int_X c_s(\Omega_X) = (-1)^{\text{dim}X} \chi(X).
		\end{align*}
		If we use that
		\begin{align*}
			\frac{\Lambda_{-1}\mathcal{A}^\vee}{\Lambda_{-y^{-1}}\mathcal{A}^\vee} = y^{\text{rk}\mathcal{A}} \frac{\Lambda_{-1}\mathcal{A}}{\Lambda_{-y}\mathcal{A}},
		\end{align*}
		we can write analogously
		\begin{align*}
			(-y^{\frac{1}{2}})^{\text{dim}\mathcal{A}}\chi_y &=  \chi\left(Y, \Lambda_{-1} \left(y \Omega_Y - y^{-1}\mathcal{A}^\vee + \mathcal{A}^\vee\right) \right)\\
			&=  y^{\text{rk}(\mathcal{A})}\chi\left(Y, \Lambda_{-1} \left(y \Omega_Y - y\mathcal{A} + \mathcal{A}\right) \right)\\
			&= y^{\text{rk}(\mathcal{A})} \chi\left(X, \Lambda_{-y} \Omega_X \right).
		\end{align*}
		The result for the Hirzebruch genus follows from
		\begin{align*}
			\text{dim}\mathcal{A} = \text{dim}X + 2 \text{rk}(\mathcal{A}) \cong_2 \text{dim}X.
		\end{align*}
		For the elliptic genus we know, set $T_\mathcal{A}^{\text{vir}}:= T_A - y \Omega_\mathcal{A}$, that
		\begin{align*}
			\text{Ell}(q,y) = \chi\left(Y, \mathcal{E}_{1/2}(T_\mathcal{A}^{\text{vir}}) \otimes \sqrt{\text{det}(T_\mathcal{A}^{\text{vir}})^\vee} \otimes \Lambda_{-1} (T_\mathcal{A} \otimes \mathfrak{s} - \mathcal{N}^\vee_{Y/A})\right).
		\end{align*}
		Since $T_\mathcal{A}$ restricts onto $X$ as $T_X + y \mathcal{A} + \mathcal{A}^\vee$ we find 
		\begin{align*}
			(T^{\text{vir}}_\mathcal{A})_{\vert X} = T_X - y \Omega_X
		\end{align*}
		and hence
		\begin{align*}
			\text{Ell}(q,y) = y^{\frac{\text{dim}X}{2}}\chi\left(Y, \mathcal{E}_{1/2}(T_Y - y \Omega_X) \otimes K_X \otimes \Lambda_{-y^{-1}} T_X \right)
		\end{align*}
		Since $\mathcal{E}_{1/2}$ is multiplicative and satisfies $\mathcal{E}_{1/2}(V) = \mathcal{E}_{1/2}(V^\vee)$ we have
		\begin{align*}
			\mathcal{E}_{1/2}(T_Y - y \Omega_X) = \mathcal{E}_{1/2}(T_Y - y^{-1} T_X) = \mathcal{E}(T_X).
		\end{align*}
		Moreover $K_X \otimes \Lambda_{-y^{-1}} T_X = (-y^{-1})^{\text{dim}X} \Lambda_{-y} \Omega_X$, hence 
		\begin{align*}
			\text{Ell}(q,y) = (-y^{\frac{1}{2}})^{-\text{dim}X}\chi\left(Y, \mathcal{E}(T_X) \otimes \Lambda_{-y} \Omega_X \right).
		\end{align*}
	\end{proof}
	\subsection{Example: complete intersections in projective spaces.}
		In the projective space $\mathbb{P}^n$ consider, for $d_1, ..., d_m \in \mathbb{N}$, the complete intersection $X$ of $m<n$ hypersurfaces $H_i$ of degree $d_i$. We can compute the Euler number of $X$ using theorem \ref{mainTheorem} with the data of
		\begin{align*}
			V \simeq \mathbb{C}^{n+1} \times \mathbb{C}^m,
		\end{align*}
		with a $G \times {\mathbb{C}^\ast} \simeq \mathbb{C}^\ast \times \mathbb{C}^\ast$-action given by
		\begin{align*}
			(t,s)\cdot (x, y) := (tx, st^{-d_1} y_1, ..., st^{-d_m}y_m).
		\end{align*}
		If $\mathfrak{t} \simeq \mathbb{C}$ is the Lie algebra of $T=G$ and $u$ is the coordinate, we choose the linearization $\mathcal{L} \rightarrow V$ given by 1-dimensional $T$-representation corresponding, via remark \ref{linearizationsAndCharacters}, to the stability $\xi:= u$:
		\begin{align*}
			T \curvearrowright \mathbb{C} \quad : \quad t \cdot z := t^{-1} z.
		\end{align*}
		The quotient is the total space of a vector bundle over $\mathbb{P}^n$:
		\begin{align*}
			V/\!/G \simeq \text{Tot}\left(\bigoplus_{i=1}^m \mathcal{O}(-d_i)\right).
		\end{align*}
		The weights of the $T$-action are $\mathcal{W}= \lbrace u, -d_1 u, \dots, -d_m u \rbrace$ and the affine functions $f_i : \mathfrak{t} \rightarrow \mathbb{C}$, indexed by $i \in \lbrace 1, ..., m+n+1 \rbrace$, are
		\begin{align*}
			f_i(u) = \begin{cases}
				u & \text{if } i \leq n+1\\
				1-d_{i-n-1} u & \text{otherwise}
			\end{cases} 
		\end{align*}
		having union of zero sets $\mathfrak{M}:= \lbrace 0 \rbrace \cup \lbrace 1/d_j \text{ : } 1\leq j\leq m \rbrace$.
		The only stable intersection here is the origin, so $\mathfrak{M}^\xi=\lbrace 0\rbrace$, corresponding to the base $\mathbb{P}^n$ of the bundle.  
		The function $Z_s$ for $d=1$ and $s=1$ is
		\begin{align*}
			Z(u) = \left(\frac{1-u}{u}\right)^{n+1}
			\prod_{i=1}^{m}\frac{d_i u}{1-d_i u}.
		\end{align*}
		In this case the residue $\text{JK}^{\lbrace u \rbrace}_0$ coincides with the usual residue $\text{res}_{u=0}$ and thus our formula predicts
		\begin{align*}
			(-1)^{n-m} \chi(X) &= \left(\prod_{i=1}^{m} d_i\right) \text{res}_{u=0} \left( u^{m-n-1}
			\frac{(1-u)^{n+1}}{\prod_{i=1}^m(1-d_i u)}\right)
		\end{align*}
		Notice that this is nothing but the formula for $\int_X c_\bullet(\Omega_X)$ once we use the conormal sequence
		\begin{align*}
			0 \rightarrow \bigoplus_{i=1}^m \mathcal{O}(-d_i) \rightarrow \Omega_{\mathbb{P}^n} \rightarrow \Omega_X \rightarrow 0
		\end{align*}
		to write, with respect to $i : X \hookrightarrow \mathbb{P}^n$, the pushforward
		\begin{align*}
			i_\ast(c_\bullet(\Omega)_X) = \left(\prod_{i=1}^{m} d_i h\right) \frac{(1-h)^{n+1}}{\prod_{i=1}^m(1-d_i h)}.
		\end{align*}
	\subsection{Example: a Calabi-Yau 3-fold in G(2,4).}\label{CY3}
	Here we give an example where $G$ is not abelian. Consider the embedding of $G(2,4)$ via the Pl\"ucker embedding, so that $\mathcal{O}(1) \simeq \text{det}(S)^\vee$, where $S$ is the tautological subbundle. By adjunction, a generic section of $\mathcal{O}(4)$ cuts a smooth CY 3-fold $X$ whose Euler number $\chi(X)=-176$ has been computed for example in \cite{Ito}. We now find the same number by using theorem \ref{mainTheorem}. The total space of the bundle $\mathcal{O}(-4)$ can be built as the quotient of 
		\begin{align*}
			V:= \text{Mat}_{2 \times 4}(\mathbb{C}) \times \mathbb{C}
		\end{align*}
		by the action of $G:= GL_2(\mathbb{C})$ given by
		\begin{align*}
			g \cdot (M, z) := (M g^{-1}, \text{det}(g)^4 z).
		\end{align*}
		Here the maximal subtorus is $T \simeq (\mathbb{C}^\ast)^2$ has Lie algebra $\mathfrak{t} \simeq \mathbb{C}^2$ and we denote its coordinates with $u_1, u_2$. The Weyl group is $W \simeq \mathbb{Z}_2$ and acts by exchanging the coordinates. If we choose the Weyl invariant stability $\xi := -u_1-u_2$ we find that the linearization is given via remark \ref{linearizationsAndCharacters} by the character
		\begin{align*}
			G \curvearrowright \mathbb{C} \quad : \quad g \cdot x := \text{det}(g) x,
		\end{align*} 
		hence the quotient is indeed
		\begin{align*}
			V/\!/G \simeq \text{Tot}(\mathcal{O}(-4)).
		\end{align*}
		Notice that the $\mathbb{C}^\ast$-action scaling the fibers of this bundle is induced from the $R$-charge assigning 0 to the piece $\text{Mat}_{2\times 4}(\mathbb{C})$ of $V$ and $1$ to $\mathbb{C}\subset V$.
		The set $\mathcal{W}$ of weights for the $G$-action on $V$ is 
		\begin{align*}
			\lbrace -u_1, -u_2, 4u_1+4u_2 \rbrace
		\end{align*}
		and the weight space decomposition is
		\begin{align*}
			V \simeq \mathbb{C}^4 \times \mathbb{C}^4 \times \mathbb{C}.
		\end{align*}
		The roots $\mathfrak{R}$ of $G$ are the functionals $\pm(u_1-u_2)$, hence the function $Z:=Z_{s=1}$ for $d=1$ is 
		\begin{align*}
			Z = \frac{u_1-u_2}{(1+u_1-u_2)} \frac{u_2-u_1}{(1+u_2-u_1)} 
			\left(\frac{1 +u_1}{-u_1}\right)^4
			\left(\frac{1 +u_2}{-u_2}\right)^4
			\frac{- 4u_1 - 4u_2}{1 + 4u_1 + 4u_2}.
		\end{align*}
		The poles of $Z$ are drawn in the following picture: the full lines correspond to the poles coming from the weights while the dashed ones come from the roots:
		\[
		\begin{tikzpicture}
			\begin{axis}[
				axis lines = box,
				enlarge x limits=0.05,
				enlarge y limits=0.05,
				every axis x label/.style={at={(1,0)},anchor=north west},
				every axis y label/.style={at={(0,1.1)},anchor=north},
				xlabel = {$u_1$},
				ylabel= {$u_2$},
				xmin=-1.5,
				xticklabels={$-\frac{3}{2}$, $-1$, $-\frac{1}{2}$, $0$, $\frac{1}{2}$, $1$, $\frac{3}{2}$},
				xtick = {-1.5, -1, -0.5, 0, 0.5, 1, 1.5},
				xmax=1.5,			
				ymin=-1.5,
				yticklabels={$-\frac{3}{2}$, $-1$, $-\frac{1}{2}$, $0$, $\frac{1}{2}$, $1$, $\frac{3}{2}$},
				ytick = {-1.5, -1, -0.5, 0, 0.5, 1, 1.5},
				ymax=1.5]
				
				\addplot [
				domain=-1.25:1.25, 
				samples=100, 
				samples y=0,
				color=black,
				]
				({-x}, {0});
				
				\addplot [
				domain=-1.25:1.25, 
				samples=100, 
				samples y=0,
				color=black,
				]
				({0}, {-x});
				
				\addplot [
				domain=-1.25:0.75, 
				samples=100, 
				samples y=0,
				color=black,
				]
				({x}, {-x-1/4});
				
				% We draw the roots
				\addplot [
				domain=-0.25:1.25, 
				samples=100, 
				samples y=0,
				color=black,
				dashed
				]
				({x}, {x-1});
				
				\addplot [
				domain=-1.25:0.25, 
				samples=100, 
				samples y=0,
				color=black,
				dashed
				]
				({x}, {x+1});
				
				%We add the intersections of weights hyperplanes
				
				\addplot[
				color=black,
				mark=*,
				]
				coordinates {
					(-0.25,0)(0, -0.25)(0,0)
				};
			\end{axis}
		\end{tikzpicture}
		\]
		As it's clear from the picture there are 3 isolated intersections of the hyperplane arrangement defined by the weights:
		\begin{align*}
			\mathfrak{M} = \left\lbrace \left(0,0\right), \left(-\frac{1}{4}, 0\right), \left(0,-\frac{1}{4}\right) \right\rbrace.
		\end{align*}
		Notice that, as shown in the picture, no hyperplane defined by the roots passes from these points, hence condition (\ref{InvertibilityEulerClasses}) is satisfied. Let's check which one of these intersections is stable.
		\begin{itemize}
			\item The weights corresponding to the hyperplanes vanishing at the origin $(0,0)$ are $-u_1$ and $-u_2$. Since $\xi = -u_1 - u_2$ is in the positive span of these vectors, the origin is stable.
			\item The weights corresponding to the hyperplanes vanishing at the point $(-\frac{1}{4},0)$ are $4u_1+4u_2$ and $-u_2$. Since 
			\begin{align*}
				\xi = -u_1-u_2 = -\frac{1}{4} (4u_1+4u_2),
			\end{align*}
			this point is not stable.
			\item The weights corresponding to the hyperplanes vanishing at the point $(0,-\frac{1}{4})$ are $4u_1+4u_2$ and $-u_1$. Since 
			\begin{align*}
				\xi = -u_1-u_2 = -\frac{1}{4} (4u_1+4u_2),
			\end{align*}
			again this point is not stable.
		\end{itemize}
		We have finally shown that $\mathfrak{M}^\xi = \lbrace (0,0) \rbrace$, hence
		\begin{align*}
			-\chi(X) = \frac{1}{2} \text{JK}^{\lbrace-u_1, -u_2\rbrace} (Z, \tilde{\xi})
		\end{align*}
		by proposition \ref{completeIntInvariants} and theorem \ref{mainTheorem}, where we had to deform $\xi$ into a sum-regular stability $\tilde{\xi}$. We chose $\tilde{\xi}:= -\frac{1}{10}(11 u_1 + 9 u_2)$. The lattice $\gamma$ spanned by the weights is $\text{span}_{\mathbb{Z}}(u_1, u_2)$ and the corresponding linear form $d\mu$ is $u_1 \wedge u_2$. There are two possible flags of $\mathfrak{t}^\vee$ that we can generate with the two vectors $\lbrace -u_1, -u_2 \rbrace$ and here we compute their contribution to the JK residue:
		\begin{itemize}
			\item Consider first the flag $F_1 := 0 \subset \text{span}_{\mathbb{C}}(-u_1)\subset \mathfrak{t}$. The corresponding basis $\kappa$ is given by $\kappa_1 = -u_1$ and $\kappa_2 = -u_1-u_2$, so
			\begin{align*}
				\tilde{\xi} = \frac{2}{10} \kappa_1+\frac{9}{10}\kappa_2
			\end{align*}
			and the flag is stable. The contribution of this flag is by definition
			\begin{align*}
				&\left\vert \frac{d\mu}{\kappa_1 \wedge \kappa_2} \right \vert \text{Res}_{\kappa_2=0} \text{Res}_{\kappa_1=0} Z(u_1, u_2) = \text{Res}_{\kappa_2=0} \text{Res}_{\kappa_1=0} Z(-\kappa_1, \kappa_1-\kappa_2).
			\end{align*}
			Notice that $Z(-\kappa_1, \kappa_1-\kappa_2)$ is
			\begin{align*}
				\frac{\kappa_2-2\kappa_1}{(1+\kappa_2-2\kappa_1)} \frac{2\kappa_1-\kappa_2}{(1+2\kappa_1-\kappa_2)} 
				\left(\frac{1 -\kappa_1}{\kappa_1}\right)^4
				\left(\frac{1 +\kappa_1-\kappa_2}{\kappa_2-\kappa_1}\right)^4
				\frac{4\kappa_2}{1 - 4\kappa_2}.
			\end{align*}
			and, after some computations, taking the residue first at $\kappa_1 = 0$ and then at $\kappa_2=0$ gives $352$ as result.
			\item The second flag is $F_2 := 0 \subset \text{span}_{\mathbb{C}}(-u_2)\subset \mathfrak{t}$. The corresponding basis $\kappa$ is given by $\kappa_1 = -u_2$ and $\kappa_2 = -u_1-u_2$, so
			\begin{align*}
				\tilde{\xi} = -\frac{2}{10} \kappa_1+\frac{11}{10}\kappa_2
			\end{align*}
			and the flag is unstable, hence it gives no contribution to the JK residue.
		\end{itemize}
		We have finally shown that theorem \ref{mainTheorem} gives
		\begin{align*}
			-\chi(X) = \frac{1}{2} \text{JK}^{\lbrace-u_1, -u_2 \rbrace} (Z) = \frac{352}{2} = 176,
		\end{align*}
		in agreement with our expectations.
	\subsection{Virtual invariants of critical loci in quiver varieties.}\label{SubsectionQuivers}
	Given a quiver $Q$, the corresponding quiver varieties are defined as the GIT quotients of the spaces of representations by actions of products of general linear groups. Here we recall how to build these varieties, for more details see the survey \cite{Reineke}.
	Let $Q$ be a connected quiver with finitely many arrows and nodes. It can be with or without oriented cycles, with or without loops.
	The set of nodes of the quiver is denoted with $Q_0$, while the set of arrows with $Q_1$. We have two functions, called \textit{head} and \textit{tail}
	\begin{align*}
		h, t : Q_1 \rightarrow Q_0,
	\end{align*}
	which send an arrow into the node corresponding to its head or tail.
	\subsubsection{The space of representations and the gauge group.}
	Given a \textit{dimension vector} $D \in \mathbb{N}^{Q_0}$, we consider the space of $D$-dimensional representations
	\begin{align}\label{repsSpaceNA}
		V := \bigoplus_{\alpha \in Q_1} \text{Mat}_{D_{h(\alpha)}\times D_{t(\alpha)}}(\mathbb{C}).
	\end{align}
	There is a group $\prod_{v \in Q_0} \text{GL}_{D_v}(\mathbb{C})$ acting on the representation space by
	\begin{align}\label{originalAction}
		\prod_{v \in Q_0} \text{GL}_{D_v}(\mathbb{C}) \curvearrowright V \qquad : \qquad (M \cdot \Phi)_\alpha := M_{h(\alpha)} \Phi_\alpha M_{t(\alpha)}^{-1}.
	\end{align}
	The diagonal $\Delta \simeq \mathbb{C}^\ast$ inside this group acts trivially so, in order to work with an effective action, we consider the action of the projectivized group:
	\begin{align*}
		G := \left(\prod_{v \in Q_0} \text{GL}_{D_v}(\mathbb{C})\right)/\Delta.
	\end{align*}
	The maximal subtorus of $G$ is the quotient of the group of tuples of diagonal matrices by $\Delta$:
	\begin{align*}
		T = \left(\prod_{v \in Q_0}(\mathbb{C}^\ast)^{D_v}\right)/\Delta.
	\end{align*}
	The Lie algebras $\mathfrak{t} \subseteq \mathfrak{g}$ are the quotients of
	\begin{align*}
		\bigoplus_{v \in Q_0} \mathbb{C}^{D_v} \subseteq \bigoplus_{v \in Q_0} \mathfrak{gl}_{D_v}(\mathbb{C})
	\end{align*}
	by the diagonal subspace $\text{span}_\mathbb{C}(\mathbb{1})$.
	The Weyl group of $G$ is
	\begin{align*}
		W = \prod_{v \in Q_0} \mathfrak{S}_{D_v}
	\end{align*}
	and acts on $\mathfrak{t}$ by permuting the components in each piece $\mathbb{C}^{D_v}$.
	\begin{lem}\label{rootsOfG}
		The roots of $G$, as functionals on $\mathfrak{t}$, are induced by the functionals
		\begin{align*}
			\beta^v_{j,i} : \bigoplus_{v \in Q_0} \mathbb{C}^{D_v} \rightarrow \mathbb{C} \qquad \beta^v_{j,i}(u):= u_{v, j} - u_{v, i}
		\end{align*}
		where $v \in Q_0$ and $i, j \in \lbrace 1, ..., D_v\rbrace$.
	\end{lem}
	\begin{proof}
		It is well known from the structure of roots of general linear groups that these $\beta^v_{j, i}$ are the roots for the group before taking the quotient by $\Delta$. Since they are all zero on the diagonal subspace $\text{span}_\mathbb{C}(\mathbb{1})$, they descend onto $\mathfrak{t}_\mathbb{C}$ and define the roots of $G$.
	\end{proof}
	It's also easy to describe the weights for the action of $T$ on $V$:
	\begin{lem}
		The weights of the $T$-representation $V$, whose set we denote with $\mathcal{W}$, are induced by the functionals
		\begin{align*}
			\rho^\alpha_{j,i} : \bigoplus_{v \in Q_0} \mathbb{C}^{D_v} \rightarrow \mathbb{C} \qquad \rho^\alpha_{j,i}(u):= u_{h(\alpha), j} - u_{t(\alpha), i}
		\end{align*}
		where $\alpha \in Q_1$, $i \in \lbrace 1, ..., D_{t(\alpha)}\rbrace$ and $j \in \lbrace 1, ..., D_{h(\alpha)}\rbrace$. Moreover, two arrows $\alpha, \beta \in Q_1$ define the same weights if and only if they share the same head and tail.
	\end{lem}
	\begin{proof}
		The induced $T$-action on the irreducible piece $\text{Mat}_{D_{h(\alpha)}\times D_{t(\alpha)}}(\mathbb{C}) \subseteq V$ is
		\begin{align*}
			\left(t \cdot \Phi\right)_{j,i} := t_{h(\lambda), j} \Phi_{j,i} t^{-1}_{t(\lambda), i}
		\end{align*}
		and the thesis follows from the action being diagonal.
	\end{proof}
	\subsubsection{Linearizations and stabilities.}
	Since $V$ is an affine space, the linearizations for this actions corresponds to choices of a character of $G$, namely an element of
	\begin{align*}
		\text{Hom}(G, \mathbb{C}^\ast) \simeq (\mathfrak{t}_\mathbb{Z}^\vee)^W.
	\end{align*}
	It's easy to see that
	\begin{align*}
		(\mathfrak{t}_\mathbb{Z}^\vee)^W \simeq  V\left( \sum_v D_v \xi_v \right) \subset \mathbb{Z}^{Q_0}
	\end{align*}
	and that the correspondence above is
	\begin{align*}
		\mathbb{Z}^{Q_0} \supset V\left( \sum_v D_v \xi_v \right) \xrightarrow{\sim} \text{Hom}(G, \mathbb{C}^\ast) \qquad : \qquad \xi \rightarrow \prod_{v \in Q_0} \text{det}^{\xi_v}.
	\end{align*}
	\begin{rem}
		Recall our convention established in remark \ref{linearizationsAndCharacters}: in order to preserve the Kempf-Ness correspondence, we define the linearization $\mathcal{L} \simeq V \times \mathbb{C}$ induced by a character $\chi : G \rightarrow \mathbb{C}^\ast$ as
		\begin{align*}
			g \cdot (\Phi, z) := (g\cdot \Phi, \chi(g)^{-1}z).
		\end{align*}
	\end{rem}
	As shown in \cite{Reineke}, if $\xi$ is a regular stability the action on the semistable locus is free and the corresponding GIT quotient $\mathcal{A}:= V/\!/G$ is smooth. In the literature, $\mathcal{A}$ is called a \textit{quiver variety} and it is often denoted with $\mathcal{M}_D^{\xi\text{-sst}}(Q)$.
	\subsubsection{The circle action.}
	We can endow $V$ with an action of $\mathbb{C}^\ast$ by choosing an element $R \in \mathbb{Z}^{Q_1}$ (the R-charge) and writing
	\begin{align*}
		(s \cdot \Phi)_\alpha := s^{R_\alpha} \Phi_\alpha \qquad \forall \alpha \in Q_1.
	\end{align*}
	Notice that, since $G$ acts linearly on each irreducible piece $\text{Mat}_{D_{h(\alpha)}\times D_{t(\alpha)}}(\mathbb{C})$ of $V$, then the actions of $G$ and $\mathbb{C}^\ast$ commute.
	Now we want to study under which conditions the fixed subvariety $\mathcal{A}^{\mathbb{C}^\ast}$ is proper (or compact). We have a simple description, proven in \cite{Procesi}, of the $G$-invariant functions on the space of representations
	\begin{theorem}\label{generatorsInvariantQuiver}
		The ring of $G$-invariant functions on $V$ is generated by the functions
		\begin{align*}
			f_\gamma : V \rightarrow \mathbb{C} \quad : \quad f_{\gamma}(\Phi) := \text{Tr}\left(\prod_{\alpha \in \gamma} \Phi_\alpha\right)
		\end{align*}
		where $\gamma$ ranges over the set of oriented cycles of $Q$ of length less or equal than $\left(\sum_{v \in Q_0} D_v\right)^2$.
	\end{theorem}
	The characterization of the generators of the ring of invariant functions given by the theorem above, together with remark \ref{projectivityOfInt}, allows us to find a condition to impose on the $R$-charge in order to have a projective fixed subvariety $\mathcal{A}^{\mathbb{C}^\ast}$:
	\begin{cor}
		Assume that the following hypothesis holds:
		\begin{align}\label{HpCompactQuiver}
			\text{for every oriented cycle } \gamma \text{ in } Q, \text{ we require } \sum_{\alpha \in \gamma} R_{\alpha} \neq 0.
		\end{align}
		Then the fixed locus $\mathcal{A}^{\mathbb{C}^\ast}$ is projective.
	\end{cor}
	\subsubsection{The formula.}
	Consider the following affine functions defined on $\mathfrak{t}$ by the weights of the $T$-action and the $R$-charge:
	\begin{align*}
		f^\alpha_{j,i} : \mathfrak{t} \rightarrow \mathbb{C} \quad : \quad f^\alpha_{j,i}(u) := \rho^\alpha_{j,i}(u)+R_\alpha = u_{h(\alpha), j}-u_{t(\alpha), i} + R_\alpha,
	\end{align*}
	indexed by $\alpha\in Q_1$, $i \in \lbrace 1, ..., D_{t(\alpha)}\rbrace$ and $j \in \lbrace 1, ..., D_{h(\alpha)}\rbrace$.
	We will denote with $I$ the set of such indices:
	\begin{align*}
		I := \coprod_{\alpha \in Q_1} \lbrace 1, ..., D_{t(\alpha)}\rbrace \times \lbrace 1, ..., D_{h(\alpha)}\rbrace.
	\end{align*}
	The set $\mathfrak{M} \subset \mathfrak{t}$ of isolated intersections (see definition \ref{isolatedInt}) of the hyperplanes $V(f^\alpha_{j,i})$ is the set of points $P \in \mathfrak{t}$ at which at least $\text{dim}\mathfrak{t}$ independent hyperplanes vanish:
	\begin{align*}
		\mathfrak{M} = \left\lbrace P \in \mathfrak{t} \,\, : \,\, \exists I_P \subseteq I \text{ s.t. } \bigcap_{i \in I_P} V(f_i) = \lbrace P \rbrace \right\rbrace.
	\end{align*}
	Given $P \in \mathfrak{M}$, we set  $I_P := \lbrace i \in I \,\, : \,\, f_i(P)=0 \rbrace$ and let $\mathcal{W}_P$ be the set of the corresponding weights:
	\begin{align*}
		\mathcal{W}_P = \left\lbrace \rho^\alpha_{j, i} \in \mathcal{W} \quad : \quad (\alpha, i, j) \in I \right\rbrace.
	\end{align*}
	In other words, a weight $\rho^\alpha_{j,i}$ belongs to $\mathcal{W}_P$ if and only if $\rho^\alpha_{j,i}(P) + R_i = 0$.
	The subset $\mathfrak{M}^\xi \subseteq \mathfrak{M}$ of stable isolated intersections (see definition \ref{stableIsolatedInt}) with respect to the regular stability $\xi$ is
	\begin{align*}
		\mathfrak{M}^\xi= \left\lbrace P \in \mathfrak{M} \,\, : \,\, \xi \in \text{span}_{\mathbb{R}_{\geq 0}}(\mathcal{W}_P) \right\rbrace.
	\end{align*}
	We can specialize theorem \ref{mainTheorem} to the case of quiver varieties to obtain the following theorem, appeared for the first time in the physics literature (see \cite{BMP}).
	Denoting with $\vert D \vert$ the sum $\sum_{v \in Q_0} D_v$, consider the following meromorphic functions $\mathfrak{t} \dashrightarrow \mathbb{C}$. First we have the more general
	\begin{align*}
		Z_{\tau,\hbar} := &\left(\frac{2\pi \hbar \, \eta(\tau)^3}{\theta(\tau \vert d\hbar)}\right)^{\vert D \vert-1} \prod_{v \in Q_0} \prod_{\substack{i, j = 1\\i \neq j}}^{D_v} \frac{\theta(\tau \vert \hbar(u_{v,j}-u_{v,i}))}{\theta(\tau \vert \hbar(d+ u_{v,i} - u_{v,j}))}\\
		&\times \prod_{\alpha \in Q_1} \prod_{i = 1}^{D_{t(\alpha)}}\prod_{j = 1}^{D_{h(\alpha)}} \frac{\theta(\tau \vert \hbar(d + u_{t(\alpha),i} - u_{h(\alpha),j} -R_{\alpha}))}{\theta(\tau \vert \hbar(u_{h(\alpha),j} - u_{t(\alpha),i} +R_{\alpha}))},
	\end{align*}
	then we consider its limit for $\tau \rightarrow i \infty$
	\begin{align*}
		Z_\hbar := &\left(\frac{\pi \hbar}{\sin(d\pi \hbar)}\right)^{\vert D \vert-1} \prod_{v \in Q_0} \prod_{\substack{i, j = 1\\i \neq j}}^{D_v} \frac{\sin(\pi \hbar(u_{v,j}-u_{v,i}))}{\sin(\pi \hbar(d+ u_{v,i} - u_{v,j}))}\\
		&\times \prod_{\alpha \in Q_1} \prod_{i = 1}^{D_{t(\alpha)}}\prod_{j = 1}^{D_{h(\alpha)}} \frac{\sin(\pi \hbar(d + u_{t(\alpha),i} - u_{h(\alpha),j} -R_{\alpha}))}{\sin(\pi \hbar(u_{h(\alpha),j} - u_{t(\alpha),i} +R_{\alpha}))}
	\end{align*}
	and finally the limit $\hbar \rightarrow 0$
	\begin{align*}
		Z =\left(\frac{1}{d}\right)^{\vert D \vert-1} \prod_{v \in Q_0} \prod_{\substack{i, j = 1\\i \neq j}}^{D_v} \left(\frac{u_{v,j}-u_{v,i}}{d+u_{v,i}-u_{v,j}}\right)\\
		\times \prod_{\alpha \in Q_1} \prod_{i = 1}^{D_{t(\alpha)}}\prod_{j = 1}^{D_{h(\alpha)}} \frac{d-R_{\alpha}+u_{t(\alpha),i}-u_{h(\alpha),j}}{R_{\alpha} +u_{h(\alpha),j} - u_{t(\alpha),i}}.
	\end{align*}
	The following result holds true:
	\begin{theorem}\label{mainTheoremQuivers}
		Let $D \in \mathbb{N}^{Q_0}$ be a dimension vector for a quiver $Q$. Assume the stability $\xi$ is regular, $\tilde{\xi}$ is a sum-regular perturbation and the $R$-charge $R \in \mathbb{Z}^{Q_1}$ sums to a nonzero number over each oriented cycle. Consider the hyperplanes in $\mathfrak{t}$ defined by the functionals
		\begin{align*}
			f_{\alpha_{i,j}}(u) := u_{h(\alpha),i} - u_{t(\alpha), j} + R_{\alpha}
		\end{align*}
		defined for every $\alpha \in Q_1$, $i \in \lbrace1, ..., D_{h(\alpha)}\rbrace$, $j \in \lbrace1, ..., D_{t(\alpha)}\rbrace$.
		Assume that, for every stable isolated intersection $P \in \mathfrak{M}^\xi$ of such hyperplane arrangement, the inequality
		\begin{align*}
			P_{v,i} - P_{v,j} + d \neq 0
		\end{align*}
		holds for every $v \in Q_0$ and $i,j \in \lbrace 1, ..., D_v\rbrace$.
		The virtual invariants of a critical locus of a degree $d$ potential defined on the quiver moduli space $\mathcal{A}= V/\!/G = \mathcal{M}_D^{\xi\text{-sst}}(Q)$ are
		\begin{align*}
			\text{DT} &= \frac{1}{\prod_{v \in Q_0}(D_v!)} \sum_{P \in \mathfrak{M}^\xi} \text{JK}_P^{\mathcal{W}_P}\left(Z_s,\tilde{\xi}\right),\\
			\chi_{e^{2 \pi i \hbar}} &= \frac{1}{\prod_{v \in Q_0}(D_v!)} \sum_{P \in \mathfrak{M}^\xi} \text{JK}_P^{\mathcal{W}_P}\left(Z_\hbar, \tilde{\xi}\right),\\
			\text{Ell}(e^{2 \pi i \tau}, e^{2\pi i \hbar}) &= \frac{1}{\prod_{v \in Q_0}(D_v!)} \sum_{P \in \mathfrak{M}^\xi} \text{JK}_P^{\mathcal{W}_P}\left(Z_{\tau, \hbar}, \tilde{\xi}\right).
		\end{align*}
	\end{theorem}
	\begin{rem}
		It's worth remarking that in the case of quiver varieties, the procedure of pulling back integrals from $V/\!/G$ onto $V/\!/T$ given by Martin's formula performed in section \ref{subsPullbackOnAbelian} pulls back integrals from the variety corresponding to the quiver $Q$ to the one corresponding to the quiver $\hat{Q}$ obtained by blowing up the nodes. This procedure consists in replacing a node of dimension vector $d$ with $d$ distinct nodes of dimension 1 and connecting two nodes in $\hat{Q}$ with an arrow if and only if the original nodes were connected in $Q$. For example, if we start from the quiver
		\[
		\begin{tikzcd}
			\mathbb{C} \arrow[r, bend right = 30]\arrow[r, bend left = 30] & \mathbb{C}^3 \arrow[r] & \mathbb{C}
		\end{tikzcd},
		\]
		the blowup is
		\[
		\begin{tikzcd}
			& \mathbb{C} \arrow[rd] & \\
			\mathbb{C} \arrow[r, bend right = 10]\arrow[r, bend left = 10]\arrow[ru, bend right = 10]\arrow[ru, bend left = 10] \arrow[rd, bend right = 10]\arrow[rd, bend left = 10] & \mathbb{C} \arrow[r] & \mathbb{C}\\
			& \mathbb{C} \arrow[ru] &
		\end{tikzcd}.
		\]
	\end{rem}
	\subsection{Example: DT invariants of $\mathbb{A}^3$.}
		Consider the quot scheme of quotient sheaves of $\mathcal{O}_{\mathbb{A}^3}^{\oplus r}$ having finite length $n$:
		\begin{align*}
			X^n_r := \text{Quot}_{\mathbb{A}^3}(\mathcal{O}^{\oplus r}, n).
		\end{align*}
		As shown in \cite{Ricolfi}, this scheme is closely related to the moduli space of representations of the quiver 
		\[
		\begin{tikzcd}
			\mathbb{C}^N \arrow[out=70,in=110,loop,swap,"\alpha_1"]
			\arrow[out=160,in=200,loop,swap,"\alpha_2"]
			\arrow[out=250,in=290,loop,swap,"\alpha_3"] & & \mathbb{C} \arrow[ll, bend right = 30, swap, "\beta_1"] \arrow[ll, "\dots"]
			\arrow[ll, bend left = 30, "\beta_r"]
		\end{tikzcd}
		\]
		The space of representations is 
		\begin{align*}
			\text{Rep}(Q) \simeq \text{Mat}_{n\times n}(\mathbb{C})^{\oplus 3} \oplus \left(\mathbb{C}^N\right)^{\oplus r}
		\end{align*}
		and the gauge group $G:=\text{GL}_N(\mathbb{C})$ acts on it by
		\begin{align*}
			g \cdot (A_1, A_2, A_3, b_1, \dots, b_r) := (gA_1 g^{-1}, gA_2 g^{-1}, gA_3 g^{-1}, g b_1, \dots, g b_r).
		\end{align*}
		The stability $\xi := 1 \in \mathbb{Z} \simeq \text{Hom}(G,\mathbb{C}^\ast)$ is regular and the corresponding quotient is a smooth nonproper variety $\mathcal{A} := \text{Rep}(Q)/\!/G$. The $G$-invariant function 
		\begin{align*}
			\tilde{\varphi} : \text{Rep}(Q) \rightarrow \mathbb{C} \quad : \quad \tilde{\varphi}(A_1, A_2, A_3, b_1, ..., b_r) := \text{Tr}\left(A_1 [A_2, A_3]\right)
		\end{align*}
		defines a potential $\varphi : \text{Quot}^n_r \rightarrow \mathbb{C}$ whose critical locus is the quot scheme:
		\begin{align*}
			X^n_r = \text{Crit}(\varphi) \subseteq \mathcal{A}.
		\end{align*}
		We can enrich the picture with a $\mathbb{C}^\ast$-action on $\text{Rep}(Q)$ given by
		\begin{align*}
			s \cdot (A_1, A_2, A_3, b_1, ..., b_r) := (s^{R_1}A_1, s^{R_2}A_2, s^{R_3}A_3, b_1, ..., b_r)
		\end{align*}
		with $R_1, R_2, R_3 \geq 1$. Theorem \ref{generatorsInvariantQuiver} ensures that the fixed locus $\mathcal{A}^{\mathbb{C}^\ast}$ is projective and, moreover, $\varphi$ is homogeneous of degree $d=R_1+R_2+R_3$. The equivariant DT invariant of $X^n_{r}$ in the sense of Definition \ref{DT} is called \textit{$n^{\text{th}}$ degree zero cohomological DT invariant of rank $r$} of $\mathbb{A}^3$:
		\begin{align*}
			\text{DT}^n_r(\mathbb{A}^3) = \int_{[X^n_r]^{\text{vir}}} 1.
		\end{align*}
		Our formula can, in principle, be used to compute these invariants (this approach to compute the invariants of $\mathbb{A}^3$ is present in the physics literature, in particular in \cite{BeniniTanzini}). The function $Z : \mathbb{C}^N \dashrightarrow \mathbb{C}$ of which we have to extract the residues is
		\begin{align*}
			Z(u_1, ..., u_N) := d^{-N} \prod_{\substack{i,j =1\\ i \neq j}}^N \frac{u_i-u_j}{d+u_i-u_j} \prod_{k=1}^N \left(\frac{d-u_k}{u_k}\right)^r \prod_{l=1}^3 \prod_{m, n=1}^N \frac{d-R_l-u_n + u_m}{R_l + u_n - u_m},
		\end{align*}
		where $d= R_1+R_2+R_3$. It seems a hard task to compute the residues appearing in the formula, but one can check numerically that for small values of $n$ the formula confirms equality
		\begin{align*}
			\sum_{n=0}^\infty DT^n_r(\mathbb{A}^3) q^n = M((-1)^r q)^{-r \frac{(R_1+R_2)(R_1+R_3)(R_2+R_3)}{R_1 R_2 R_3}},
		\end{align*}
		which has been proven in \cite{Ricolfi}. Here $M$ is the \textit{MacMahon function}, the generating functions of plane partitions:
		\begin{align*}
			M(q) = \prod_{k=1}^\infty \frac{1}{(1-x^k)^k}.
		\end{align*}
	\appendix
	\section{Completion of proofs.}
	\begin{lem}\label{trivialCircle}
		Consider a torus representation $(\mathbb{C}^\ast)^N \curvearrowright \mathbb{C}^M$. If the weights of the action don't span the Lie algebra of the torus, then there is a $\mathbb{C}^\ast$ contained in $T$ acting trivially.
	\end{lem}
	\begin{proof}
		Let $a_1, ..., a_M \in \mathbb{Z}^N$ be the weights for the given action, so that
		\begin{align*}
			(t \cdot x)_i := \prod_{j=1}^N t_i^{a_{i,j}} x_i \qquad \forall i \in \lbrace 1, ..., M \rbrace.
		\end{align*}
		By hypothesis $\text{span}_\mathbb{Q}(a_1, ..., a_M)$ is a strict subspace of $\mathbb{Q}^N$, hence we can consider an orthogonal vector $b \in \mathbb{Q}^N$ with respect to the Euclidean scalar product $\langle \cdot , \cdot \rangle$ and we can rescale it to make it integral and primitive. Since this vector defines a ray of rational slope in the Lie algebra, we can consider the corresponding $\mathbb{C}^\ast \hookrightarrow T$ given by $s \rightarrow (s_1^{b_1}, ..., s_N^{b_N})$ and this acts trivially on $\mathbb{C}^M$ since $\langle b, a_i \rangle = 0$ for every $i$.
	\end{proof}
	Throughout this section, let $T$ be a torus and let $t_1, ..., t_N$ be the corresponding 1-dimensional representations. Let $\mathfrak{s}$ be the irreducible representation of $\mathbb{C}^\ast$ having weight 1. Once we consider $\mathfrak{s}$ as an equivariant bundle over a point, let $y$ be class of $\mathfrak{s}$ in $K^0_{T\times \mathbb{C}^\ast}(\text{pt})$ and let $s := c_1^{T \times \mathbb{C}^\ast}(\mathfrak{s})$. Then
	\begin{lem}\label{computationsLambda}
		Let $V$ be a representation of $T\times {\mathbb{C}^\ast}$:
		\begin{align*}
			V = \sum_{j=0}^m t^{w_j} \otimes \mathfrak{s}^{a_j} .
		\end{align*}
		with trivial fixed part. Then
		\begin{align*}
			\text{ch}^{T \times {\mathbb{C}^\ast}}\left(\frac{\Lambda_{-y^d} V}{\Lambda_{-1}V}\right)  = (e^{\frac{ds}{2}})^{\text{dim}V} \prod_{j =0}^m \frac{\sinh(\frac{w_j+(a_j-d)s}{2})}{\sinh(\frac{w_j - a_js}{2})}.
		\end{align*}
	\end{lem}
	\begin{proof}
		We have that
		\begin{align*}
			\text{ch}^{T \times {\mathbb{C}^\ast}}\left(\frac{\Lambda_{-1} V \otimes \mathfrak{s}^d}{\Lambda_{-1}V}\right) &= \prod_{j =0}^m \frac{1-e^{(d-a_j)s-w_j}}{1-e^{-w_j-a_js}}\\
			&= \prod_{j =0}^m e^{\frac{ds}{2}} \frac{e^{-\frac{ds}{2}}-e^{(\frac{d}{2}-a_j)s-w_j}}{1-e^{-w_j-a_js}}\\
			&= \prod_{j =0}^m e^{\frac{ds}{2}} \frac{e^{\frac{w_j + (a_j-d)s}{2}}-e^{\frac{(d-a_j)s-w_j}{2}}}{e^{\frac{w_j+a_js}{2}}-e^{\frac{-w_j-a_js}{2}}}\\
			&= e^{\frac{(m+1) ds}{2}} \prod_{j =0}^m \frac{\sinh\left(\frac{w_j + (a_j-d)s}{2}\right)}{\sinh\left(\frac{w_j+a_js}{2}\right)}.
		\end{align*}
	\end{proof}
	The classes $\mathcal{E}_{1/2}$ and $\hat{\mathcal{E}}_{1/2}$ introduced in (\ref{tildeEclass}) take a particularly simple form when we work over a point, as we now describe. 
	\begin{lem}\label{computationsE}
		Let $V$ be a virtual representation of a torus $T$:
		\begin{align*}
			V = \sum_{j=0}^m t^{w_j} - \sum_{k=0}^n t^{z_k}.
		\end{align*}
		with trivial fixed part. Then
		\begin{align*}
			\hat{\mathcal{E}}_{1/2}(V) = \left(-i q^{\frac{1}{12}}\eta(q)\right)^{\text{dim}V} \frac{\prod_{k=0}^n\theta(q; t^{z_k})}{\prod_{j=0}^m\theta(q; t^{w_j})}.
		\end{align*}
		If $W$ is a trivial, 1-dimensional representation of $T$:
		\begin{align*}
			\mathcal{E}_{1/2}(W) = \frac{q^{\frac{1}{12}}}{\eta(q)^2}
		\end{align*}
	\end{lem}
	\begin{proof}
		Notice that for every weight $\mu \in \mathbb{Z}^N$
		\begin{align*}
			\text{Sym}_{q^n} t^\mu = \sum_{k=0}^\infty (q^n t^\mu)^k = \frac{1}{1-q^n t^\mu},
		\end{align*}
		hence the contribution of the summand $t^\mu$ to $\hat{\mathcal{E}}_{1/2}(V)$ is 
		\begin{align*}
			\frac{t^{-\frac{\mu}{2}}}{1-t^{-\mu}} \prod_{n\geq 1}  \frac{1}{(1-q^n t^\mu)(1-q^n t^{-\mu})} = -i q^{\frac{1}{12}}\frac{\eta(q)}{\theta(q; t^\mu)}.
		\end{align*}
		The second statement directly follows from the definition.
	\end{proof}
	Here we explain how, in section \ref{proof}, one can reduce to considering the case in which $\rho(P)\in \mathbb{Z}$ is for every stable intersection $P \in \mathfrak{M}^\xi$ and every weight $\rho \in \mathcal{W}$.
	Let $S:= \mathbb{C}^\ast$, consider an action $S \curvearrowright \mathcal{A}$ on a projective variety and, given $k \in \mathbb{Z}$, consider the action of $S^\prime:= S$ on $\mathcal{A}$ induced by the morphism $S^\prime \ni s \mapsto s^k \in S$. Since from a $S$-equivariant structure we can produce a $S^\prime$-equivariant one, we have a morphism
	\begin{align*}
		\phi_k : K_S^0(\mathcal{A}) \rightarrow K_{S^\prime}^0(\mathcal{A}).
	\end{align*}
	If we denote with $y$ the $K$-theory class of the trivial bundle with $S$-action on the fiber with weight one and with $y^\prime$ its $S^\prime$ analogue, $\phi_k$ sends $y \mapsto (y^\prime)^k$. Moreover the equivariant Euler characteristics fit into a commutative square having as horizontal rows the morphisms $\phi_k$ for $\mathcal{A}$ and for the point. If $\varphi : \mathcal{A} \rightarrow \mathbb{C}$ is a degree $d$-function with respect to the $S$-action then it is of degree $dk$ with respect to $S^\prime$, and the $K$-theory classes of the equivariant virtual tangent bundles of its critical locus are related by $\phi_k$. In particular this proves that
	\begin{lem}\label{reductionIntegralCase}
		If we denote with the prime symbol $\prime$ the invariants of $\text{Crit}(\varphi)$ for the $S^\prime$-action, we have that
		\begin{align*}
			&\text{DT} = \text{DT}^\prime, \quad \chi_{y^k} = \chi^\prime_y,\quad \text{Ell}(q,y^k) = \text{Ell}^\prime(q,y).
		\end{align*}
	\end{lem} 
	In particular, if we express these invariants in terms of $\hbar$ and $\tau$:
	\begin{align*}
		&\text{DT} = \text{DT}^\prime, \quad \chi_{e^{2\pi i \hbar}} = \chi^\prime_{e^{2 \pi i \frac{\hbar}{k}}},\quad \text{Ell}(e^{2\pi i \tau},e^{2 \pi i \hbar}) = \text{Ell}^\prime(e^{2\pi i \tau},e^{2 \pi i \frac{\hbar}{k}}).
	\end{align*}
	Going back to our problem, we can pick a $k \in \mathbb{N}$ so that $k \rho(P) \in \mathbb{Z}$ for all stable intersections and all weights, and apply theorem \ref{mainTheorem} for the $S^\prime$-action. Notice that considering this action corresponds to considering the scaled $R$-charges $R^\prime := k R$. Let's consider, for example, the case of the elliptic genus. It's immediate to see that the new set of stable intersections $\mathfrak{M}^{\xi, \prime}$ for the $S^\prime$-action coincides with the scaling $k \mathfrak{M}^\xi$ of the original set of stable intersections for the $S$-action. Moreover, for every $P \in \mathfrak{M}^\xi$, the corresponding weights satisfy $\mathcal{W}_{kP} = \mathcal{W}_P$. This means that the contribution of the point $kP$ to $\text{Ell}(e^{2\pi i \tau},e^{2 \pi i \hbar})=\text{Ell}^\prime(e^{2\pi i \tau},e^{2 \pi i \frac{\hbar}{k}})$ is given by
	\begin{align*}
		\text{JK}_{kP}^{\mathcal{W}_P} \left(\left(\frac{2\pi \frac{\hbar}{k} \, \eta(\tau)^3}{\theta(\tau \vert d\hbar)}\right)^{\text{dim}T} \prod_{\alpha \in \mathfrak{R}} \frac{\theta(\tau\vert \frac{\hbar}{k} \alpha)}{\theta(\tau\vert \frac{\hbar}{k}(kd + \alpha))} \prod_{i=1}^{\text{dim}V} \frac{\theta(\tau\vert \frac{\hbar}{k}(kd- kR_i - \rho_i))}{\theta(\tau\vert \frac{\hbar}{k}(kR_i + \rho_i))}\right).
	\end{align*}
	By using lemma \ref{rescalingJK} we can compute the residue after rescaling the coordinates on $\mathfrak{t}$ by a factor of $k$, so the quantity above is
	\begin{align*}
		\text{JK}_{P}^{\mathcal{W}_P} \left(\left(\frac{2\pi \hbar \, \eta(\tau)^3}{\theta(\tau \vert d\hbar)}\right)^{\text{dim}T} \prod_{\alpha \in \mathfrak{R}} \frac{\theta(\tau\vert \hbar \alpha)}{\theta(\tau\vert \hbar(d + \alpha))} \prod_{i=1}^{\text{dim}V} \frac{\theta(\tau\vert \hbar(d- R_i - \rho_i))}{\theta(\tau\vert \hbar(R_i + \rho_i))}\right),
	\end{align*}
	hence we recover the formula for the chiral elliptic genus of $\text{Crit}(\varphi)$ with respect to the $S$-action:
	\begin{align*}
		\text{Ell}(e^{2\pi i \tau},e^{2 \pi i \hbar}) = \frac{1}{\vert W \vert}\sum_{P \in \mathfrak{M}^\xi} \text{JK}^{\mathcal{W}_P}_P (Z_{\tau, \hbar}).
	\end{align*}
	The cases of the DT invariant and the virtual Hirzebruch genus are completely analogous.
	\section{Some Mathematica code.}
	The procedure described by theorem \ref{mainTheorem} to compute virtual invariants of critical loci can be written down into an algorithm. Here we try to describe a \underline{naive} approach one could take. First of all, assume we are given a linear space $\mathbb{C}^N$ with a $G$-action and a $\mathbb{C}^\ast$-action. Assume that the maximal subtorus $T\simeq (\mathbb{C}^\ast)^k$ of $G$ is split and that the $T \times \mathbb{C}^\ast$-action is diagonal. The dual Lie algebra of $T$ is now identified with $\mathbb{C}^k$ and we pick a regular stability $\xi \in \mathbb{Z}^k$. The dual Lie algebra of $\mathbb{C}^\ast$ is identified with $\mathbb{C}$ and we can consider the weights for the $T \times \mathbb{C}^\ast$-action, which are couples of the form $(w, R)$ for $w \in \mathbb{Z}^k, R \in \mathbb{Z}$. Let $\text{m}(w,R)$ be the dimension of the corresponding weight space. We can encode this data in the case of example \ref{CY3} in the following way:
	\[
	\includegraphics[width=\textwidth, keepaspectratio]{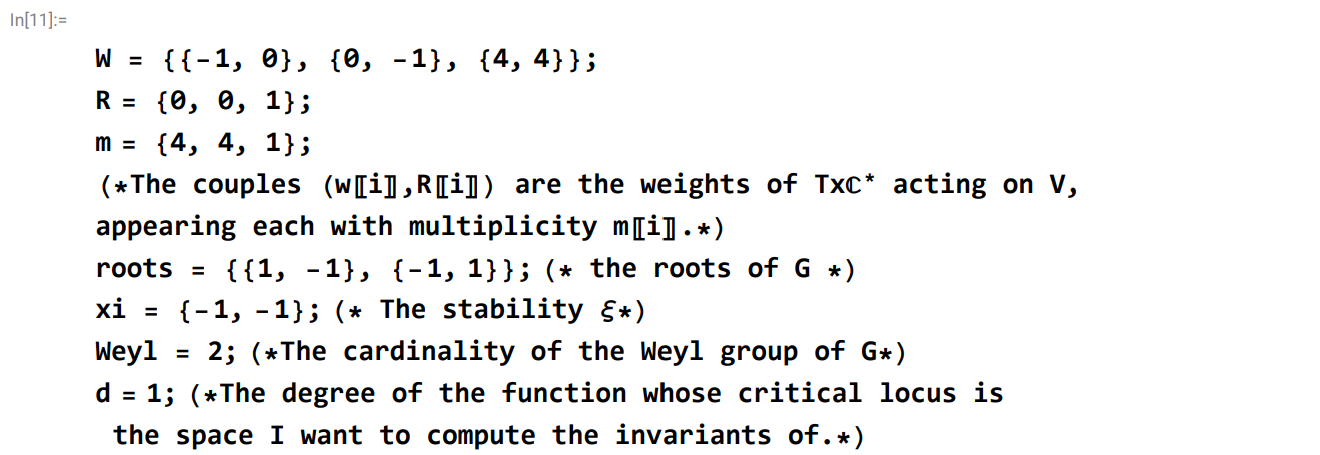}
	\]
	By invoking the function below, one can compute the virtual invariants of the critical locus of a degree $d$ function. 
	\[
	\includegraphics[width=\textwidth, keepaspectratio]{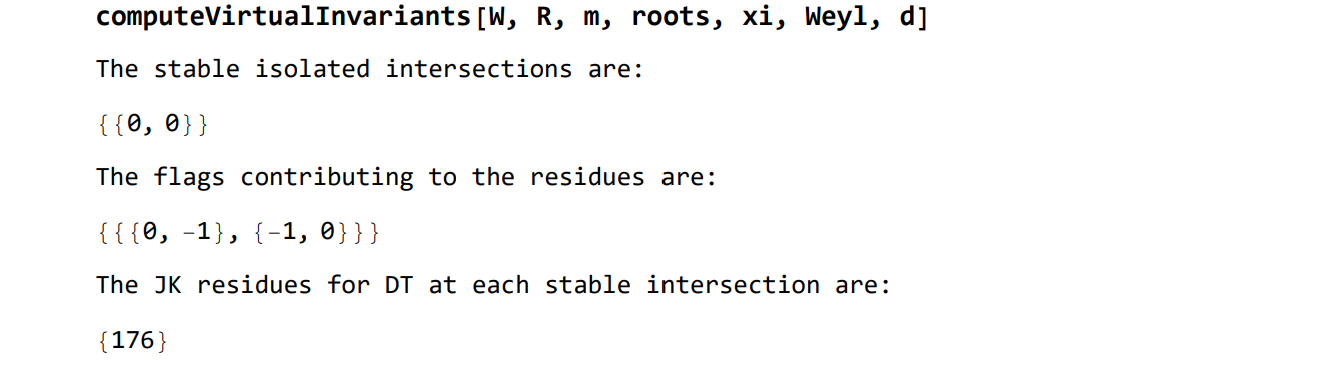}
	\]
	\[
	\includegraphics[width=\textwidth, keepaspectratio]{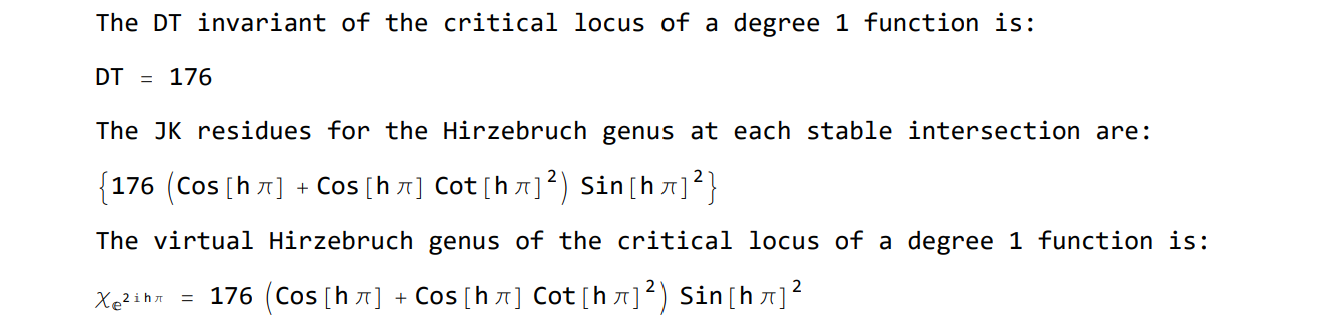}
	\]
	The following function expresses the Hirzebruch genus in terms of $y$ instead of $\hbar$.
	\[
	\includegraphics[width=\textwidth, keepaspectratio]{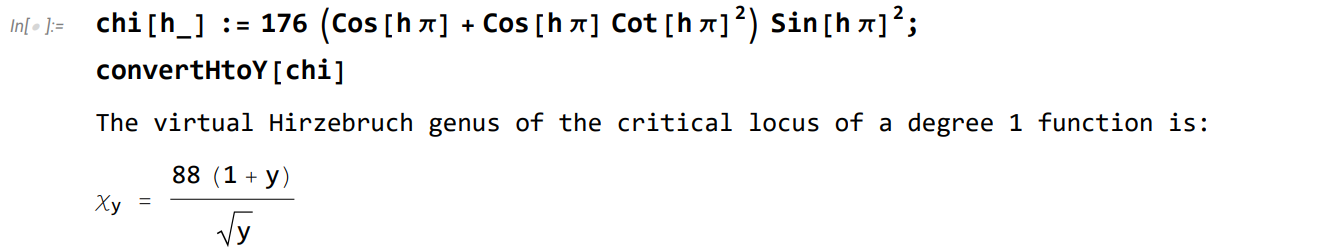}
	\]
	The Mathematica \cite{Mathematica} code is freely available on the \href{https://sites.google.com/view/riccardoontani/research}{author's website}.
	\printbibliography %Prints bibliography
\end{document}